\documentclass[10pt,leqno]{amsart}

   \usepackage[centertags]{amsmath}
   \usepackage{amsfonts}
   \usepackage{amsmath}
   \usepackage{amssymb}
   \usepackage{amsthm}
   \usepackage{newlfont}
   \usepackage[latin1]{inputenc}
   \usepackage[notcite,notref]{showkeys}

\allowdisplaybreaks[3]

\theoremstyle{plain}
\newtheorem{theorem}{Theorem}[section]
\newtheorem{lemma}[theorem]{Lemma}

\newtheorem{corollary}[theorem]{Corollary}

\theoremstyle{definition}

\theoremstyle{remark}
\newtheorem{remark}[theorem]{Remark}
\newtheorem*{remark*}{Remark}

\numberwithin{equation}{section}

\newcommand\D{{\mathcal D}}
\newcommand\A{{\mathcal A}}

\newcommand\CC{{\mathbb C}}

\newcommand\RR{{\mathbb R}}
\newcommand\ZZ{{\mathbb Z}}
\newcommand\NN{{\mathbb N}}
\newcommand\PP{{\mathbb P}}

\newcommand\pp{\mbox{$\mathfrak{p}_{F}$}}

\newcommand\ps{\mbox{$\mathfrak{p}_{-c-F}$}}

\newcommand\dgr{\operatorname{dgr}}

\newcommand\Sh{\mbox{\Large $\mathfrak {s}$}}

\newcommand*\pFqskip{8mu}
\catcode`,\active
\newcommand*\pFq{\begingroup
        \catcode`\,\active
        \def ,{\mskip\pFqskip\relax}%
        \dopFq
}
\catcode`\,12
\def\dopFq#1#2#3#4#5{%
        {}_{#1}F_{#2}\biggl(\genfrac..{0pt}{}{#3}{#4};#5\biggr)%
        \endgroup
}

\title[$\D$-operators, orthogonal polynomials and difference equations]
{Using $\D$-operators to construct orthogonal polynomials satisfying  higher order difference or differential equations}

\author{Antonio J. Dur\'an}
\address{A. J. Dur\'an \\
Departamento de An\'{a}lisis Matem\'{a}tico \\
Universidad de Sevilla \\
Apdo (P. O. BOX) 1160\\
41080 Sevilla. Spain.}
\email{duran@us.es }
\thanks{Partially supported by MTM2012-36732-C03-03 (Ministerio de Economía y Competitividad),
FQM-262, FQM-4643, FQM-7276 (Junta de Andalucía) and Feder Funds (European
Union).}

\subjclass{33C45, 33E30, 42C05}

\keywords{Differential and difference operators; classical orthogonal polynomials; classical discrete orthogonal polynomials;
Charlier polynomials; Meixner polynomials; Krawtchouk polynomials; Hahn polynomials;
Laguerre polynomials; Jacobi polynomials; Krall polynomials.}

   \date{}
   
\begin{document}
   \maketitle

   \begin{abstract} We introduce the concept of $\D$-operators associated to a sequence of polynomials $(p_n)_n$ and an algebra $\A$ of operators acting in the linear space of polynomials. In this paper, we show that this concept is a powerful tool to generate families of
orthogonal polynomials which are eigenfunctions of a higher order difference or differential operator. Indeed, given a classical discrete family $(p_n)_n$ of orthogonal polynomials (Charlier, Meixner, Krawtchouk or Hahn), we form a new sequence of polynomials $(q_n)_n$ by considering a linear combination of two consecutive $p_n$: $q_n=p_n+\beta_np_{n-1}$, $\beta_n\in \RR$. Using the concept of $\D$-operator, we determine the structure of the sequence $(\beta_n)_n$ in order that the polynomials $(q_n)_n$ are common eigenfunctions of a higher order difference operator. In addition, we generate sequences $(\beta _n)_n$ for which the polynomials $(q_n)_n$ are also orthogonal with respect to a measure. The same approach is applied to the classical families of Laguerre and Jacobi polynomials.
\end{abstract}

\section{Introduction and results}
The issue of orthogonal polynomials $(p_n)_n$ which are common eigenfunctions of a difference operator
goes back at least for one century and a half. The first example was introduced by Chebyshev in 1858. Along the first decades of the XX century,
some other examples appeared and they are associated with the names of Charlier, Meixner, Krawtchouk and Hahn.
In 1941, O.T. Lancaster \cite{La} classified all
orthogonal polynomials on the real line satisfying second order difference equations of the form
$$
\sigma \Delta\nabla p_n+\tau \nabla p_n=\lambda _np_n,\quad n\ge 0,
$$
where $\sigma$ and $\tau$ are polynomials of degree at most $2$ and $1$, respectively (independent of $n$), and  $\Delta $ and $\nabla$
denote the
first order difference operators:
\begin{equation}\label{defdn}
\Delta (f)=f(x+1)-f(x),\quad \quad \nabla (f)=f(x)-f(x-1)
\end{equation}
(that is, the polynomials $(p_n)_n$ are common eigenfunctions of the second order difference operator
$\sigma \Delta\nabla +\tau \nabla$). If we consider orthogonal polynomials with respect to positive measures (on the real line), the only solutions
are the polynomials of Charlier, Meixner, Krawtchouk and Hahn.
For similarity with the families of Hermite, Laguerre and Jacobi,
they are usually called classical \textit{discrete} polynomials. The adjective \textit{discrete} means that each of these families is
orthogonal with respect to a \textit{discrete} measure. More precisely, the discrete orthogonalizing weights for these families of polynomials are
($N\in \NN \setminus \{0\}$):
\begin{align}\label{chwi}
&\mbox{Charlier} &\sum _{x\in \NN}&\frac{a^x}{x!}\delta _x, \quad 0<a;\\\label{mewi}
&\mbox{Meixner} &\sum _{x\in \NN}&\frac{a^x\Gamma (x+c)}{x!}\delta _x, \quad 0<a<1, 0<c;\\ \label{krwi}
&\mbox{Krawtchouk}  &\sum _{x= 0}^{N-1}&\frac{a^x}{\Gamma(N-x)x!}\delta _x, \quad 0<a;\\ \label{hawi}
&\mbox{Hahn} &\sum _{x= 0}^{N-1}&\frac{\Gamma (\alpha -x)\Gamma (x+c)}{\Gamma (N-x)x!}\delta_x, \quad \alpha >N-1, c>0 .
\end{align}

However, almost nothing was known about orthogonal polynomials which are common eigenfunctions of a higher order difference operator.
We will expand a finite difference operator using the basis formed by the shift operators $\Sh_{l}$, $l\in \ZZ$, defined by
\begin{equation}\label{defsh}
\Sh _l(f)=f(x+l)
\end{equation}
(this basis is equivalent to the basis formed by the powers of the difference operators $\Delta $ and $\nabla$ defined by (\ref{defdn})).
Hence we  consider finite difference operators of the form
\begin{equation}\label{doho}
D=\sum_{l=s}^rf_l\Sh _l, \quad s\le r, s,r\in \ZZ.
\end{equation}
If $f_r,f_s\not =0$, the order of $D$ is then $r-s$. We also say that $D$ has genre $(s,r)$.

Until very recently no examples of orthogonal polynomials which are common eigenfunctions of a higher order difference operator
of the form (\ref{doho}) were known,
if we except the \textit{discrete} classical families themselves.
What we know for the continuos or $q$ cases (differential or $q$-difference operators, respectively)
has seemed to be of little help because adding Dirac deltas to the classical discrete families
seems not to work (see the papers \cite{BH} and \cite{BK} by Bavinck, van Haeringen and Koekoek answering, in the negative,
a question posed by R. Askey in 1991).

The first examples of orthogonal polynomials $(p_n)_n$ which are common eigenfunctions of a difference operator of order bigger than $2$
have been introduced by the author
in \cite{du1}. We generate the orthogonalizing measures for these families of polynomials multiplying the classical discrete weights by
certain variants of the annihilator polynomial of a set of numbers. For a finite set of numbers $F$, this annihilator polynomial $\pp$ is
defined by
\begin{equation}\label{dp1i}
\pp(x)=\prod _{j\in F}(x-j).
\end{equation}
Using this  annihilator polynomial (and other of its variants), we have posed in \cite{du1} a number of conjectures. Two of them are the
following. Write $n_F$
for the number of elements of $F$.

\textbf{Conjecture 1.} Let $\rho $ be either the Meixner (\ref{mewi}) or the Hahn weight
(\ref{hawi}) (in this last case we assume $N$ big enough to avoid trivialities).
For a finite set $F$ of positive integers, consider the polynomial $\ps $ defined by
$$
\ps (x)=\prod_{j\in F}(x+c+j),
$$
where $c$ is the parameter in the Meixner and Hahn weights.
Then, the sequence (finite or infinite) of orthogonal polynomials with respect to the measure $\rho _{F,c}=\ps \rho$ are common eigenfunctions
of a difference operator $D$ of the form
$$
D=\sum _{l=-r}^{r}f_{l}(x)\Sh_l, \quad \mbox{where $r=\displaystyle \sum _{x\in F}x-\frac{n_F(n_F-1)}{2}+1.$}
$$

\textbf{Conjecture 2.}
Let $\rho $ be any of the classical discrete weights of Charlier,
Meixner, Krawtchouk or Hahn (in this last cases we assume $N$ big enough to avoid trivialities). For a finite set $F$ of positive integers,
assume that the measure $\rho _F=\pp \rho$ has associated a sequence $(p_{n})_n$ of orthogonal polynomials (finite or infinite: see
the preliminaries for more details), where $\pp$ is the polynomial defined by (\ref{dp1i}).
Then, these orthogonal polynomials $(p_{n})_n$ are common eigenfunctions
of a difference operator $D$ of the form
$$
D=\sum _{l=-r}^{r}f_{l}(x)\Sh_l, \quad \mbox{where $r=\displaystyle \sum _{x\in F}x-\frac{n_F(n_F-1)}{2}+1.$}
$$

\bigskip

The purpose of this paper is to introduce a method to generate
sequences of orthogonal polynomials which are eigenfunctions of a higher order difference operator. The method is based on the concept of
$\D$-operator associated to a sequence of polynomials
$(p_n)_n$ and an algebra $\A$ of operators acting in the linear space of polynomials.
Starting from one of the classical discrete families $(p_n)_n$ of orthogonal polynomials,
we form a new sequence of polynomials $(q_n)_n$ by considering a linear combination of two consecutive $p_n$:
$q_n=p_n+\beta_np_{n-1}$, $\beta_n\in \RR$. Using the concept of $\D$-operator associated to the family $(p_n)_n$ and the algebra of
difference operators with polynomial coefficients, we determine the structure of the sequence $(\beta_n)_n$ in order that the polynomials
$(q_n)_n$ are commons eigenfunctions of a higher order difference operator.
We will use this technique to prove the conjectures 1 and 2 above when $F$ is the
finite set formed by the first $k$ consecutive
positive integers $F=\{ 1,2,\cdots, k\}$, $k\ge 1$. This method also generates sequences of polynomials which are eigenfunctions of higher order differential (see the Appendix of this paper) or $q$-difference operators (see \cite{AD}). Hence, the concept of $\D$-operator
provides an unified approach to construct orthogonal polynomials which are eigenfunctions for each one of these
three different types of operators (differential, difference or $q$-difference).

The content of this paper is as follows.

In Section 3, we consider the key tool of our method:
$\D$-operators associated to
a sequence of polynomials $(p_n)_n$ ($p_n$ of degree $n$) and an algebra of operators $\A$ acting in the linear space of polynomials.
We have different types of $\D$-operators. The simplest $\D$-operator is defined from
an arbitrary sequence of numbers $(\varepsilon_n)_n$ as follows: we consider first the operator $\zeta :\PP \to \PP $ defined by linearity
from $\zeta (p_n)=\varepsilon _np_{n-1}$. Using that $\dgr(\zeta (p))<\dgr (p)$ for all polynomial $p$, we define
the operator  $\D$ by
\begin{align}\nonumber
\D &:\PP \to \PP \\\label{defToi}
\D (p)&=\sum _{j=1}^\infty (-1)^{j+1}\zeta^j(p).
\end{align}
We then say that $\D$ is a $\D$-operator for the sequence $(p_n)_n$ and the algebra $\A$ if $\D \in \A$.

For the examples in this paper we consider two algebras $\A$. When the sequence of polynomials $(p_n)_n$ is one of the classical discrete families (Sections 4, 5, 6 and 7) i.e. Charlier, Meixner, Krawtchouk or Hahn polynomials,
we consider the algebra $\A$ formed by all finite order difference
operators $\sum_{j=s}^rf_j\Sh _j$ where $f_j$ is a polynomial, $j=s,\cdots, r$:
\begin{equation}\label{algdiff}
\A =\left\{ \sum_{j=s}^rf_j\Sh _j : f_j\in \PP, j=s,\cdots ,r, s\le r \right\}.
\end{equation}
When the sequence of polynomials $(p_n)_n$ is either the Laguerre or Jacobi polynomials (Appendix), we consider the algebra $\A$ formed by all finite order differential
operators $\sum_{j=0}^kf_j(d/dx)^j$ where $f_j$ is a polynomial with degree at most $j$, $j=0,\cdots, k$:
\begin{equation}\label{algdiff1}
\A =\left\{ \sum_{j=0}^kf_j\left(\frac{d}{dx}\right)^j : f_j\in \PP, \deg(f_j)\le j, j=0,..,k, k\in \NN \right\}.
\end{equation}

Here it is an example of $\D$-operator.
Consider the Charlier polynomials $(c_{n}^{a})_n$ (see (\ref{defchp}) below), $a\not =0$, and the algebra $\A$ given by (\ref{algdiff}).
Then the sequence $\varepsilon_n=1$, $n\ge 0$, generates a $\D$-operator for
the Charlier polynomials and this algebra $\A$; indeed, the operator $\D$ defined above (\ref{defToi}) from
the sequence $\varepsilon_n=1$, $n\ge 0$, and the Charlier polynomials is equal to the Nabla operator (see (\ref{defdn})): $\D =\nabla\in \A$
(see Lemma \ref{lTch}).
In this paper, we will also display two $\D$-operators for Meixner and Krawtchouk polynomials (see Lemmas \ref{lTme}, \ref{lTkr}), four
$\D$-operators for Hahn polynomials (see Lemma \ref{lTha}), one for Laguerre polynomials (see Lemma \ref{lTlag}) and two for Jacobi polynomials (see Lemma \ref{lTjac}).
For $\D$-operators associated to $q$-classical families of orthogonal polynomials see \cite{AD}.

Roughly speaking, our method consists in the following. In addition to the polynomials $(p_n)_n$ and the algebra $\A$, we
also have an operator $D_p\in \A$ for which the polynomials $(p_n)_n$ are eigenfunctions.
We denote $(\theta_n)_n$ the corresponding sequence of eigenvalues: $D_p(p_n)=\theta_np_n$. To simplify
we consider here in the Introduction the case when $\theta_n$ is linear in $n$.
Assume also that we have identified a sequence $(\varepsilon_n)_n$ which defines
a $\D$-operator for the sequence of polynomials $(p_n)_n$ and the algebra $\A$.
We can then produce a large class
of  polynomial families $(q_n)_n$ which are eigenfunctions of operators in the algebra $\A$ proceeding as follows.
Take an arbitrary polynomial $Q$ (such that $Q(\theta _n)\not =0$, $n\ge 0$), and define
the sequence of numbers $(\beta_n)_{n\ge 1}$  by
$$
\beta_n=\varepsilon_n\frac{Q(\theta_n)}{Q(\theta_{n-1})}.
$$
Define next a new sequence of polynomials $(q_n)_n$ by $q_0=1$ and
\begin{align}\label{qnint}
q_n=p_n+\beta_np_{n-1}, \quad n\ge 1.
\end{align}
It turns out that the sequence of polynomials $(q_n)_n$
are eigenfunctions for the operator $D_q=P(D_p)+\D Q(D_p)$, where the polynomial $P$ satisfies $P(\theta_n)-P(\theta_{n-1})=Q(\theta_{n-1})$
(since $(\theta_n)_n$
is lineal in $n$, $P$ always exists).
Since $\D$ is a $\D$-operator, $\D\in \A$ and then $D_q\in \A$ as well (see Section 3 for more details).

Only for a convenient choice of the polynomial $Q$, the polynomials $(q_n)_n$ (\ref{qnint}) are also orthogonal with respect to a measure. When the
sequence $(p_n)_n$ is any of the
classical discrete families of orthogonal polynomials, a very nice symmetry between the family $(p_n)_n$ and the polynomial $Q$ appear.
Indeed, if $p_n$, $n\ge 0$, are either the Charlier, Meixner or Krawtchouk polynomials, then
when we choose the polynomial $Q$ in the same family (but with different parameters) the polynomials $(q_n)_n$ turn out to be orthogonal as well,
while for
the Hahn polynomials, the polynomial $Q$ has to be chosen in the dual family. For other choices of the polynomial $Q$, the polynomials $(q_n)_n$
seem not to be
orthogonal with respect to any measure. In a subsequent paper, we will show that they enjoy certain orthogonality properties
which imply that the family $(q_n)_n$ has to satisfy certain  recurrence relation of order at most $2\deg(Q)+3$.
In the terminology introduced by Duistermaat and Gr\"unbaum (\cite{DG}; see also \cite{GrH1}, \cite{GrH3}), we can say that our method produces
bispectral polynomials: that is,  the family $(q_n(x))_n$ are common eigenfunctions of two  difference
operators, one in the continuous parameter $x$ and the other in the discrete parameter $n$.

Sections 4, 5, 6 and 7 are devoted to show how our method works by applying it to the families of classical discrete polynomials.
Hence, this paper together with  the forthcoming one (\cite{AD}) show
that the concept of $\D$-operator provides an unified approach to construct orthogonal polynomial
which are eigenfunctions of higher order differential, difference or $q$-difference operators, respectively.

Here it is an example of the kind of results we can get using the concept of $\D$-operator.

\begin{theorem}\label{th5.1ch} Let $P_1$ be a polynomial of degree $k+1$, $k\ge 1$,
and write $P_2(x)=P_1(x-1)-P_1(x)$ (so that $P_2$ is a polynomial of degree $k$). We assume that $P_2(-n)\not =0$, $n\ge 0$, and
define the sequences of numbers
\begin{align}\nonumber
\gamma_{n+1}&=P_2(-n),\quad n\ge 0,\\\nonumber
\lambda_n&=P_1(-n),\quad n\ge 0,\\\label{becchar}
\beta_n&=\frac{\gamma_{n+1}}{\gamma_n}\quad n\ge 1,
\end{align}
and the sequence of polynomials $q_0=1$, and for $n\ge 1$
\begin{equation}\label{qnno}
q_n(x)=c_{n}^{a}(x)+\beta_nc_{n-1}^{a}(x),
\end{equation}
where $c_{n}^{a}$, $n\ge 0, a\not =0$, are the Charlier polynomials (see (\ref{defchp}) below).
Consider the second order difference operator $D_{a}$ (\ref{sodech}) with respect to which
the Charlier polynomials are eigenfunctions. Write finally $D$ for the difference operator of order $2k+2$ and
genre $(-k-1,k+1)$
$$
D=P_1(D_{a})+\nabla P_2(D_{a}).
$$
Then $D(q_n)=\lambda_nq_n$, $n\ge 0$.
\end{theorem}

When we choose the polynomial $P_2$ in an appropriate way, the polynomials $(q_n)_n$ (\ref{qnno}) are also orthogonal with respect to
a measure. As we have pointed out above, there is a very nice symmetry in the choice of $P_2$.
 Indeed, for the Charlier case, when we choose $P_2(x)=c_{k}^{-a}(x-1)$ and assume $c_{k}^{-a}(-n)\not=0$, $n\ge -1$,
the polynomials $(q_n)_n$ (\ref{qnno}) are then orthogonal with respect
to the measure
$$
\tilde \rho_{k,a}=(-1)^kk!\delta_{-k-1}+\sum_{x=0}^\infty \frac{a^{x+k+1}}{(x+k+1)x!}\delta _x, \quad a\not =0.
$$
This provides a proof for our Conjecture 2 above for the Charlier polynomials and the finite set $F=\{ 1, 2, \cdots ,k\}$. Indeed,
the measure $\pp \rho_a$,
where $\pp$ is the polynomial (\ref{dp1i}) and $\rho _a$ is the Charlier measure (\ref{chwi}), is  equal to the
measure $\tilde \rho _{k,a}$ above shifted by $k+1$: $\pp \rho _a=\tilde\rho_{k,a}(x-k-1)$.

We have added an Appendix devoted to orthogonal polynomials $(p_n)_n$ which are eigenfunctions for a higher order differential operator of the form
\begin{equation}\label{aode}
\sum _{l=1}^kf_lp_n^{(l)}=\lambda_n p_n, \quad n\ge 0,
\end{equation}
where $f_l$, $l=1,\cdots , k$, are polynomials independent of $n$ (so that $(p_n)_n$ are common eigenfunctions
of the differential operator $D=\sum _{l=1}^kf_l\left(\frac{d}{dx}\right)^{l}$).
This issue goes back to  the late thirties, when H.L. Krall
classified all families of orthogonal polynomials which satisfy a fourth order differential equation of the form
(\ref{aode}) (see \cite{Kr1,Kr2}). Besides the classical families of Hermite, Laguerre
and Jacobi (satisfying second order differential equations), he  found three other families of orthogonal polynomials
satisfying a fourth order differential equation of this type.
These polynomials are orthogonal with respect to a positive measure which
consists of some classical weights together with a Dirac delta at the end point(s) of the interval of orthogonality:
$\chi _{[-1,1]}+M(\delta _{-1}+\delta_1)$ (Legendre type), $e^{-x}+M\delta _{0}$ (Laguerre type) and $(1-x)^\alpha \chi _{[0,1]}+M\delta _{0}$
(Jacobi type).

Since the eighties a lot of effort has been devoted to this issue.
T. H. Koornwinder (\cite{Ko}) studied  in 1984 orthogonal polynomials with respect to the Jacobi weight  together with two Dirac deltas at $\pm 1$
\begin{equation}\label{jacobid}
(1-x)^\alpha (1+x)^\beta +M\delta_{-1}+N\delta_1,\quad \alpha, \beta >-1.
\end{equation}
As a limit case he found orthogonal polynomials  with respect to the  Laguerre weight together with a Dirac delta at $0$
\begin{equation}\label{laguerred}
x^\alpha e^{-x}+M\delta_{0},\quad \alpha >-1.
\end{equation}
Some years earlier,  L.L. Littlejohn (\cite{L1,L2}) had discovered new families satisfying sixth order differential equations. They correspond
with the case $\alpha =\beta =0$ in (\ref{jacobid}) (Krall polynomials; discovered by H.L. Krall but never published)
and $\alpha =1, 2$ in (\ref{laguerred}) (Littlejohn polynomials).
A. M. Krall and L. L. Littlejohn did some work on the classification of higher order differential equations of the form (\ref{aode})
having orthogonal polynomial solutions (for several values of $k\ge 6$) (see \cite{KrL,L3}).
J. Koekoek and R. Koekoek showed in 1991 that orthogonal polynomials with respect to the weights (\ref{laguerred})
satisfy infinite order differential equations, except for nonnegative integer values of $\alpha$ for which the order reduce to
$2\alpha + 4$ (\cite{koekoe}).
Similar results were proof by R. Koekoek for the Koornwinder polynomials orthogonal with respect to the weight (\ref{jacobid}) for $\alpha =\beta$ and $M=N$ (\cite{koe}), and
later on (again with J. Koekoek) for the case $M,N>0$ (\cite{koekoe2}). K.H. Kwon and D.W. Lee proved that when we consider orthogonal polynomials
 with respect to a positive and compactly supported measure satisfying higher order differential equations, they have to be orthogonal with respect
 to some of the measures (\ref{jacobid}) (\cite{KL}, see also \cite{KLY}, \cite{KLY2}). A. Zhedanov proposed a technique to construct Krall's polynomials and found a $(2\alpha +4)$-th order differential equation for
the orthogonal polynomials with respect to (\ref{jacobid}) when $\alpha$ is a nonnegative integer and $M=0$ (\cite{Zh}).

F.A. Grünbaum and L. Haine (et al.) proved that polynomials satisfying fourth or higher order differential equations can be generated by applying Darboux transforms to certain instances of the classical polynomials (see \cite{GrH1,GrHH,GrY,H}). By modifying the Darboux process, they also introduced new families of polynomials satisfying fourth order
differential equations of the form (\ref{aode}) which also satisfy five term recurrence relations, instead of the three term recurrence relation which characterizes
the orthogonality with respect to a measure (see \cite{GrH3}). P. Iliev refined this approach and characterized the algebras of differential operators associated to these families of polynomials (see \cite{Plamen1}, \cite{Plamen2}).

In the Appendix we construct one $\D$-operator for Laguerre polynomials and two for Jacobi polynomials, and apply our method to recover some of the results commented in the previous paragraphs.

\section{Preliminaries}
For a linear operator $D:\PP \to \PP$ and a polynomial $P(x)=\sum _{j=0}^ka_jx^j$, the operator $P(D)$ is defined in the usual way
$P(D)=\sum _{j=0}^ka_jD^j$.

As usual $(a)_j$ will denote the Pochhammer symbol defined by
$$
(a)_0=1,\quad \quad (a)_j=a(a+1)\cdots (a+j-1),\quad \mbox{for $j\ge 1$, $a\in \CC$}.
$$

Let $\mu$ be a moment functional on the real line, that is, a linear mapping $\mu :\PP \to \RR$.
The $n$-th moment of $\mu $ is defined by $\mu_n=\langle \mu, x^n\rangle $.
It is well-known that any moment functional on the real line can be represented by integrating with respect to a Borel measure
(positive or not) on the real line
(this representation is not unique \cite{du0}).
If we also denote this measure by $\mu$, we have $\langle \mu,p\rangle=\int p(x)d\mu(x)$ for all polynomial $p\in \PP$. Taking this into account,
we will conveniently use along this paper one or other terminology (orthogonality with respect to a moment functional or with
respect to a measure).

We can  associate to the moment functional $\mu$ a bilinear form defined in the linear space of real polynomials:
\begin{equation}\label{bf}
B(p,q)=\langle \mu, p(x)q(x)\rangle.
\end{equation}
If the moment funcional $\mu $ is represented by a discrete measure supported in a countable set $X$ of real numbers, namely
$\mu =\sum _{x\in X}\mu(x)\delta _x,$
the bilinear form above is equal to
$B(p,q)=\sum_{x\in X} p(x)q(x)\mu(x).$

Let $K$ be a positive integer.
If
\begin{equation}\label{defth}
\Theta _n=\det \begin{pmatrix} \mu _0 & \cdots & \mu _n\\ \vdots & \ddots & \vdots \\ \mu _n & \cdots &\mu_{2n} \end{pmatrix}\not =0
\end{equation}
for $n=0,1,\cdots ,K$, we can associate to the moment functional $\mu$ a sequence of monic orthogonal polynomials $(p_n)_{n=0}^K$, $p_n$ of
degree $n$, using the formula
\begin{equation}\label{fdpo}
p_n(x)=\frac{1}{\Theta _{n-1}}\det \begin{pmatrix} \mu _0 & \mu _1&\cdots & \mu _n\\ \vdots &\vdots & \ddots & \vdots \\
\mu _{n-1} &\mu_{n}& \cdots &\mu_{2n-1}
\\1 & x&\cdots &x^n \end{pmatrix},\quad 1\le n\le K, \quad p_0=1.
\end{equation}
The polynomials $p_n$, $0\le n\le K$, satisfy that $B(p_n,x^m)=0$, for $0\le m<n$ and $B(p_n,x^n) \not =0$. Conversely, any polynomial
$q$ of degree $n$, $0\le n\le N$, satisfying $B(q,x^m)=0$, for $0\le m<n$, is equal to $p_n$ up to a multiplicative constant: $q(x)=ap_n(x)$.

The moment funcional $\mu $ can be represented by a positive measure with infinitely many points in its support if and only if $\Theta _n>0$,
for $n=0,1,\cdots $, and then the
associated bilinear form (\ref{bf}) is an inner product ($B(p,p) >0$ for all polynomial $p\not =0$).

Favard's Theorem establishes that a sequence $(p_n)_n$ of polynomials, $p_n$ of degree $n$, is orthogonal with respect to a moment functional if
and only if it satisfies
a three term recurrence relation of the form ($p_{-1}=0$)
$$
xp_n(x)=a_np_{n+1}(x)+b_np_n(x)+c_np_{n-1}(x), \quad n\ge 0,
$$
where $(a_n)_n$, $(b_n)_n$ and $(c_n)_n$ are sequences of real numbers with $a_nc_n\not =0$, $n\ge 1$. If, in addition, $a_nc_n>0$, $n\ge 1$,
then the polynomials $(p_n)_n$ are orthogonal with respect to a moment functional which can be represented by a positive measure, and conversely.

\bigskip
The kind of transformation which consists in multiplying a moment functional $\mu$ by a polynomial $r$ is called a Christoffel transform. The new
moment functional $r\mu$ is defined by $\langle r\mu,p\rangle =\langle \mu,rp\rangle $.
It has a long tradition in the context of orthogonal polynomials: it goes back a century and a half ago when
E.B. Christoffel (see \cite{Chr} and also \cite{Sz}) studied it for the particular case $r(x)=x$.

Given a moment functional $\mu$ and a polynomial $r(x)=(x-a_1)\cdots (x-a_m)$,
we can give, using \cite{Sz}, Th. 2.5, a necessary and sufficient condition for the moment functional $r\mu$
to have a sequence  of orthogonal polynomials. Indeed, write $(p_n)_{n=0}^N$ ($N$ finite or infinite) for the orthogonal polynomials
with respect to $\mu $ and $\Lambda_n$ for the square matrix
\begin{equation}\label{mata0}
\Lambda_n=(p_{n+l}(a_i))_{i=1,\cdots , m,l=0,\cdots, m-1},\quad n\le N-m+1.
\end{equation}
Then the moment functional $r\mu$ has a sequence $(q_n)_{n=0}^K$, $K\le N-m+1$, of orthogonal polynomials if and only if $\det(\Lambda_n)\not =0$
for $n=0,\cdots ,K$.

For a real number $\lambda$, the moment functional $\mu (x+\lambda)$ is defined in the usual way
$\langle \mu (x+\lambda),p\rangle =\langle \mu ,p(x-\lambda)\rangle $. Hence, if $(p_n)_n$ are orthogonal polynomials with respect
to $\mu$ then $(p_n(x+\lambda))_n$ are orthogonal with respect to $\mu(x+\lambda)$.

We will use the following Lemma to construct orthogonal
polynomials with respect to a measure together with a Dirac delta (for related results see \cite{GPV}, \cite{Y}).

\begin{lemma}\label{addel}
From a measure $\nu$ and a real number $\lambda$, we define the measure $\mu$ by $\mu=(x-\lambda)\nu$. Assume that we have a sequence $(p_n)_n$ of
orthogonal polynomials with respect to $\mu$ and write $\alpha_n=\int p_nd\nu$.
For a given real number $M$ we write $\tilde \rho =\nu + M\delta_\lambda$ and define the numbers
\begin{equation}\label{hk2}
\beta_n=-\frac{\alpha_{n}+Mp_n(\lambda)}{\alpha_{n-1}+Mp_{n-1}(\lambda)}, \quad n\ge 1
\end{equation}
(we implicitly assume that $\alpha_{n-1}+Mp_{n-1}(\lambda)\not =0$, $n\ge 1$).
Then the polynomials defined by $q_0=1$ and $q_n=p_n+\beta_np_{n-1}$, are orthogonal with respect to $\tilde \rho$.
\end{lemma}

\begin{proof}

We have to prove that $\int (x-\lambda)^jq_nd\tilde \rho =0$, $j=0,\cdots, n-1$ and $\int (x-\lambda)^nq_nd\tilde \rho \not =0$.

From the orthogonality of $(p_n)_n$ with respect to $\mu =(x-\lambda )\nu$ and the definition of $\tilde \rho$ we automatically get
for $j=1,\cdots, n-1$
$$
\int (x-\lambda)^jq_nd\tilde \rho=\int (x-\lambda)^{j-1}q_nd\mu=0,
$$
whatever the numbers $\beta_n$, $n\ge 1$, are. For $j=0$, using (\ref{hk2}), we get
\begin{align*}
\int q_nd\tilde \rho&=\int (p_n+\beta_np_{n-1})d\nu+Mp_n(\lambda)+\beta_nMp_{n-1}(\lambda)\\
&=\alpha_{n}+\beta_n\alpha_{n-1}+Mp_n(\lambda)+\beta_nMp_{n-1}(\lambda)=0.
\end{align*}
Finally, for $j=n$ we have $\int (x-\lambda)^nq_nd\tilde \rho =\beta_n\int (x-\lambda)^{n-1}p_{n-1}d\mu \not =0$.
\end{proof}

We will need the following technical Lemma.

\begin{lemma}\label{nlp} Let $u$ be a real number. We consider the polynomials $s_{0,u}=1$ and for $j\ge 1$
\begin{equation}\label{defsj}
s_{j,u}(x)=(-1)^j\prod _{i=0}^{j-1}[x+i(u-i)].
\end{equation}
Then for $j\ge 0$, $s_{j,u}(x(x-u))=(-x)_j(x-u)_j$.
\end{lemma}

\begin{proof}

If we expand  $(-x)_j(x-u)_j$, we have
$$
(-x)_j(x-u)_j=(-1)^jx\cdots(x-j+1)(x-u)\cdots (x-u+j-1).
$$
Multiplying the $(i+1)$-th and the $(j+i+1)$-th  factors in the product above $(0\le i\le j-1$), we get
$$
(x-i)(x-u+i)=x(x-u)+i(u-i).
$$
This gives for $(-x)_j(x-u)_j$ the formula
\begin{equation}\label{betjac3}
(-x)_j(x-u)_j=(-1)^j\prod_{i=0}^{j-1}[x(x-u)+i(u-i)]=s_{j,u}(x(x-u)).
\end{equation}
\end{proof}

\section{The method and its main ingredient}
In this Section, we introduce our method and its main ingredient. Although in this paper we will apply it to difference and differential operators and some of the classical discrete and classical families of
orthogonal polynomials, the method also works with  $q$-difference operators and the corresponding $q$-classical polynomials. This is the reason
why we introduce the method in an abstract setting.

The starting point is a sequence of polynomials $(p_n)_n$, $p_n$ of degree $n$, and an algebra of operators $\A $ acting in the linear space of polynomials.
In addition, we assume that the polynomials $p_n$, $n\ge 0$, are eigenfunctions of certain operator $D_p\in \A$. We write $(\theta_n)_n$ for the corresponding eigenvalues,
so that $D_p(p_n)=\theta_np_n$, $n\ge 0$.

Given a sequence of numbers $(\beta _n)_n$, we define a new sequence of polynomials $(q_n)_n$ by $q_0=1$ and
\begin{equation}\label{qncc}
q_n=p_n+\beta_np_{n-1}, \quad n\ge 1.
\end{equation}
Assuming certain structure for the sequence $(\beta _n)_n$, our procedure provides a new operator $D_q\in \A$ for which the polynomials $(q_n)_n$ are eigenfunctions.
The structure for the sequence $(\beta _n)_n$ will be determined by the main concept involved in our method: $\D$-operators.

We have two different types of $\D$-operators, depending on whether the
sequence of eigenvalues $(\theta_n)_n$ is linear in $n$ or not.

Given  a sequence of numbers $(\varepsilon_n)_n$, a $\D$-operator of type 1 associated to the algebra $\A$ and the sequence of polynomials $(p_n)_n$ is defined as follows.
We first consider  the operator $\zeta :\PP \to \PP $ defined by linearity
from $\zeta (p_n)=\varepsilon _np_{n-1}$. Using that $\dgr(\zeta (p))<\dgr (p)$ for all polynomial $p$, we next define the operator $\D$ by
\begin{align}\nonumber
\D &:\PP \to \PP \\\label{defTo}
\D(p)&=\sum _{j=1}^\infty (-1)^{j+1}\zeta^j(p).
\end{align}
We then say that $\D$ is a $\D$-operator if $\D\in \A$.

In the next Sections, we will display a $\D$-operator of type 1 for Charlier and Laguerre polynomials and two $\D$-operators of type 1 for Meixner and Krawtchouk polynomials (see Lemmas \ref{lTch}, \ref{lTme}, \ref{lTkr} and \ref{lTlag}).

Once we have introduced the main concept of $\D$-operator of type 1 defined by the sequence of numbers $(\varepsilon _n)_n$, we are ready to go on with our method.
The estructure mentioned above for the sequence $(\beta_n)_n$ (which defines the polynomials $(q_n)_n$ (\ref{qncc})) is then
$\displaystyle \beta_n =\varepsilon _n\frac{P(\theta_n)}{P(\theta_{n-1})}$, where $P$ is any arbitrary polynomial.
The details of the method are included in
the following Lemma. The Lemma is an elementary but powerful result: its proof is just an easy computation, but it provides
higher order differential, difference or $q$-difference operators for a large class of polynomials (some of which turn out
to be orthogonal with respect to a measure). For the case of differential operators, these polynomials include the Krall-Laguerre-Koornwinder families
orthogonal with respect to the measures (\ref{laguerred}).

\begin{lemma}\label{fl1v}
Let $\A$ and $(p_n)_n$ be, respectively, an algebra of operators acting in the linear space of polynomials, and a sequence of polynomials, $p_n$ of degree $n$.
We assume that $(p_n)_n$ are eigenfunctions of an operator $D_p\in \A$, that is, there exist  numbers $\theta _n$, $n\ge 0$, such that
$D_p(p_n)=\theta_np_n$, $n\ge 0$. We also have a sequence of numbers $(\varepsilon_n )_n$ which defines a $\D$-operator of type 1 for $(p_n)_n$ and $\A$ (see (\ref{defTo})).
For an arbitrary polynomial $P_2$ such that $P_2(\theta_n)\not=0$, $n\ge 0$, we define the sequences of numbers $(\gamma_n)_n$ and $(\lambda_n)_n$ by
\begin{align}
\gamma_{n+1}&=P_2(\theta_n),\quad n\ge 0, \\ \label{defjg}
\lambda_n-\lambda_{n-1}&=\gamma _n, \quad n\ge 1,
\end{align}
and assume that there exists a polynomial $P_1$ such that $\lambda_n=P_1(\theta_n)$.
We finally define the sequence of polynomials $(q_n)_n$ by $q_0=1$ and
\begin{equation}\label{defqng}
q_n=p_n+\beta_np_{n-1},\quad n\ge 1,
\end{equation}
where the numbers $\beta_n$, $n\ge 0$, are given by
\begin{equation}\label{defbetng}
\beta_n=\varepsilon_n \frac{\gamma_{n+1}}{\gamma_n}, \quad n\ge 1.
\end{equation}
Then $D_q(q_n)=\lambda _nq_n$ where the operator $D_q$ is defined by
\begin{equation}\label{defD}
D_q=P_1(D_p)+\D P_2(D_p).
\end{equation}
Moreover $D_q\in \A$.
\end{lemma}

\begin{proof}

Since $D_p(p_n)=\theta _np_n$, we have that
\begin{align*}
P_1(D_p)(p_n)&=P_1(\theta_n)p_n=\lambda_np_n,\\
P_2(D_p)(p_n)&=P_2(\theta_n)p_n=\gamma_{n+1}p_n.
\end{align*}
The definition of $D_q$ (\ref{defD}) then  gives
$$
D_q(p_n)=\lambda _np_n+\gamma_{n+1}\D (p_n).
$$
Using the definition of the $\D$-operator $\D$ (\ref{defTo}), we have
\begin{align*}
D_q(p_n)&=\lambda _np_n+\gamma_{n+1}\sum_{j=1}^n(-1)^{j+1}\zeta^j(p_n)\\
&=\lambda _np_n+\sum_{j=1}^n\gamma_{n+1}(-1)^{j+1}\varepsilon_n\cdots \varepsilon _{n-j+1}p_{n-j}.
\end{align*}
Taking into account the definition of the polynomials $(q_n)_n$ (\ref{defqng}), we get
\begin{align}\nonumber
D_q(q_n)&=D_q(p_n)+\beta_nD_q(p_{n-1})\\\nonumber &=\lambda _np_n
+\gamma_{n+1}\varepsilon _n p_{n-1}+\sum_{j=2}^n\gamma_{n+1}(-1)^{j+1}\varepsilon_n\cdots \varepsilon _{n-j+1}p_{n-j}\\\nonumber
&\quad \quad + \beta_n\lambda _{n-1}p_{n-1}+\beta_n\sum_{j=1}^{n-1}\gamma_{n}(-1)^{j+1}\varepsilon_{n-1}\cdots \varepsilon _{n-j}p_{n-1-j}\\\label{eq1}
&=\lambda_n p_n+(\beta_n\lambda_{n-1}+\gamma_{n+1}\varepsilon _n)p_{n-1}+\sum_{j=2}^{n}a_{n,j}p_{n-j},
\end{align}
where
$$
a_{n,j}=(-1)^{j+1}(\gamma_{n+1}\varepsilon_n\cdots \varepsilon _{n-j+1}-\beta_n\gamma_n\varepsilon_{n-1}\cdots \varepsilon _{n-j+1}),\quad 2\le j\le n.
$$
Using the definition of $\lambda_n$ and $\beta_n$ ((\ref{defjg}) and (\ref{defbetng})), we easily have
\begin{align*}
\beta_n\lambda_{n-1}+\gamma_{n+1}\varepsilon _n&=\lambda_n\beta_n,\\
a_{n,j}&=0,\quad 2\le j\le n.
\end{align*}
From where (\ref{eq1}) gives that $D_q(q_n)=\lambda_n p_n+\lambda_{n}\beta_np_{n-1}=\lambda_nq_n$.

Since $\D$ is a $\D$-operator for the polynomials $(p_n)_n$ and the algebra $\A$, we have that $\D\in \A$, and then $D_q\in \A$ as well.
\end{proof}

\bigskip

Actually, a slight modification of the Lemma provides a huge family of operators in $\A$ for which the polynomials $(q_n)_n$ are eigenfunctions. Indeed, consider any polynomial $G$ for which there exists a polynomial $P_1$ satisfying that $P_1(\theta _n)-P_1(\theta _{n-1})=G(\theta _{n-1})P_2(\theta _{n-1})$ (when the eigenvalues $\theta_n$ are linear in $n$, that polynomial $P_1$ always exists for each polynomial $G$). Proceeding as in the previous Lemma, it is easy to see that the operator $D_{q,G}$ defined by
$D_{q,G}=P_1(D_p)+G(D_p)\D P_2(D_p)$ belongs to $\A$ and satisfies $D_{q,G}(q_n)=P_1(\theta _n)q_n$.

\bigskip

We will apply the version of our method established in Lemma \ref{fl1v} when the eigenvalues $(\theta _n)_n$ of the polynomials $(p_n)_n$ with respect
to the operator $D_p$ are linear in $n$. It is easy to see that if $(\theta _n)_n$ are, say, quadratic in $n$, the
hypothesis of  Lemma \ref{fl1v} can not fulfil. Indeed, if $\gamma_{n+1}$ is a polynomial in $\theta _n$ and $\theta_n$ are quadratic in $n$, then $\gamma _n$
is a polynomial in $n$ of even degree, and the sequence $\lambda _n$, defined by $\lambda_n-\lambda_{n-1}=\gamma_n$, is a polynomial
in $n$ of odd degree. Hence, it can not be a polynomial in the quadratic sequence $\theta _n$.

To fix this situation we need to introduce a modified version of the $\D$-operator defined in (\ref{defTo}). Two sequences of numbers
$(\varepsilon_n)_n$ and $(\sigma_n)_n$ are now the starting point. A $\D$-operator of type 2 associated to the algebra $\A$ and the sequence of polynomials $(p_n)_n$ is defined as follows.
As before we first consider  the operator $\zeta :\PP \to \PP $ defined by linearity
from $\zeta (p_n)=\varepsilon _np_{n-1}$. We next define the operator $\D$ again by linearity from
\begin{equation}\label{defTo2}
\D (p_n)=-\frac{1}{2}\sigma_{n+1} p_n+\sum _{j=1}^n (-1)^{j+1}\sigma_{n+1-j}\zeta^j(p_n).
\end{equation}
We then say that $\D$ is a  $\D$-operator if $\D\in \A$.

When the sequence $(\sigma_n)_n$ is constant, a $\D$-operator of type 2 reduces to a $\D$-operator of type 1 (up to constants).

In Lemmas \ref{lTha} and \ref{lTjac}, we display four $\D$-operators of type 2 for
the Hahn polynomials and two $\D$-operators of type 2 for
the Jacobi polynomials, respectively.

We  can now establish the version of our method for $\D$-operators of type 2.

\begin{lemma}\label{fl2v}
Let $\A$ and $(p_n)_n$ be, respectively, an algebra of operators acting in the linear space of polynomials, and a sequence of polynomials, $p_n$ of degree $n$.
We assume that $(p_n)_n$ are eigenfunctions of an operator $D_p\in \A$, that is, there exist  numbers $\theta _n$, $n\ge 0$, such that
$D_p(p_n)=\theta_np_n$, $n\ge 0$. We also have two sequences of numbers $(\varepsilon_n )_n$ and $(\sigma_n)_n$ which define a $\D$-operator of type 2 for $(p_n)_n$ and $\A$
(see (\ref{defTo2})).
For an arbitrary polynomial $P_2$ such that $P_2(\theta_n)\not =0$, $n\ge 0$, we define the sequences of numbers $(\gamma_n)_n$ and $(\lambda_n)_n$ by
\begin{align}
\gamma_{n+1}&=P_2(\theta_n),\quad n\ge 0, \\
\lambda_n-\lambda_{n-1}&=\sigma_n\gamma _n, \quad n\ge 1,
\end{align}
and assume that there exists a polynomial $P_1$ such that $\lambda_{n+1}+\lambda_n=P_1(\theta_n)$.
We finally define the sequence of polynomials $(q_n)_n$ by $q_0=1$ and
\begin{equation}\label{defqng2}
q_n=p_n+\beta_np_{n-1},\quad n\ge 1,
\end{equation}
where the numbers $\beta_n$, $n\ge 0$, are given by
\begin{equation}\label{defbetng2}
\beta_n=\varepsilon_n \frac{\gamma_{n+1}}{\gamma_n}, \quad n\ge 1.
\end{equation}
Then $D_q(q_n)=\lambda _nq_n$ where the operator $D_q$ is defined by
\begin{equation}\label{defD2}
D_q=\frac{1}{2}P_1(D_p)+\D P_2(D_p).
\end{equation}
Moreover $D_q\in \A$.
\end{lemma}

\begin{proof}
The Lemma can be proved proceeding in the same way as in the proof of Lemma \ref{fl1v}.
\end{proof}

We remark that if an operator $\D$ is a $\D$-operator for the polynomials $(p_n)_n$ and the algebra $\A$ then it is also a $\D$-operator
for the sequence of polynomials $\tilde p_n=a_np_n$, whatever the real numbers $a_n\not =0$, $n\ge 0$, are. Indeed, it is straightforward to see that
if $\D$ is defined from $(p_n)_n$ by the sequence $(\varepsilon _n )_n$ then $\D$ is the defined from $(\tilde p_n)_n$ by the sequence
$$
\tilde \varepsilon _n=\frac{a_n}{a_{n-1}}\varepsilon_n.
$$
Using this fact, we can associate $\D$-operators to a measure having orthogonal polynomials. Indeed, we say that an operator $\D$ is a $\D$-operator
for a measure $\mu$ if it is a $\D$-operator for a sequence $(p_n)_n$ of orthogonal polynomials with respect to $\mu$.
We point out that the sequence of numbers $(\varepsilon_n)_n$, which defines the $\D$-operator from the sequence of polynomials $(p_n)_n$, depends on this sequence of polynomials. In other words, given a measure $\mu$ the corresponding set of $\D$-operators only depends on the measure $\mu$ and it is independent of the orthogonal polynomials normalization, but the associated sequences of numbers $(\varepsilon_n)_n$ depend on this normalization.

\section{Charlier case}
For $a\not =0$, we write $(c_{n}^{a})_n$ for the sequence of  Charlier polynomials
defined by
\begin{equation}\label{defchp}
c_{n}^{a}(x)=\frac{1}{n!}\sum _{j=0}^n (-a)^{n-j}\binom{n}{j}\binom{x}{j}j!
\end{equation}
(that and the next formulas can be found in \cite{Ch}, pp. 170-1; see also \cite{KLS}, pp, 247-9 or \cite{NSU}, ch. 2).
The Charlier polynomials are orthogonal  with respect to the  measure
\begin{equation}\label{chw}
\rho_a=\sum _{x=0}^\infty\frac{a^{x}}{x!}\delta _x, \quad a\not=0.
\end{equation}
The measure $\rho_a$ is positive only when $a>0$.

The three term recurrence formula for $(c_{n}^{a})_n$ is ($c_{-1}^{a}=0$)
\begin{equation}\label{repnch}
xc_n=(n+1)c_{n+1}+(n+a)c_n+ac_{n-1}, \quad n\ge 0
\end{equation}
(to simplify the notation we remove some parameters in some formulas).
They are eigenfunctions of the following second order difference operator
\begin{equation}\label{sodech}
D_{a} =x\Sh_{-1}-(x+a)\Sh_0+a\Sh_1,\qquad D_{a} (c_{n}^{a})=-nc_{n}^{a},\quad n\ge 0.
\end{equation}

They also satisfy the simple difference relation
\begin{equation}\label{sdch}
\Delta (c_n)=c_{n-1}.
\end{equation}

We first identify a $\D$-operator for the Charlier polynomials.

\begin{lemma}\label{lTch}
For $a \not =0 $, the operator $\D$ defined by (\ref{defTo}) from the sequence $\varepsilon _n=1$, $n\ge 0$,
is a $\D$-operator for the  Charlier polynomials $(c_{n}^{a})_n$ (\ref{defchp}) and the algebra $\A$  (\ref{algdiff}) of difference
operators  with polynomial coefficients.
More precisely  $\D=\nabla$.
\end{lemma}

\begin{proof}
Indeed, since $\Delta (c_{n}^{a})=c_{n-1}^{a}$ (\ref{sdch}), we have that $\zeta =\Delta$. Taking into account that
$\Delta \binom{x}{m}=\binom{x}{m-1}$, $m\ge 1$, we find
\begin{align*}
\D \binom{x}{m}&=\sum_{j=1}^\infty (-1)^{j+1}\Delta ^j\binom{x}{m}=\sum_{j=1}^m (-1)^{j+1}\binom{x}{m-j}\\
&=\binom{x}{m}-\sum_{j=0}^m (-1)^{j}\binom{x}{m-j}=\binom{x}{m}-\binom{x-1}{m}.
\end{align*}
This gives $\D =\nabla \in \A$.

\end{proof}

Lemma \ref{fl1v} and the $\D$-operator for the  Charlier polynomials found in the previous Lemma, automatically produce the large
clase of polynomials
satisfying higher order difference equations displayed in Theorem \ref{th5.1ch} in the Introduction. We are now ready to prove it.

\begin{proof}
Indeed, the theorem is a straightforward consequence of the Lemma \ref{fl1v}.
We have just to identify who the main characters are in this example.

If we write $\varepsilon_n=1$, Lemma \ref{lTch} gives that the sequence $(\varepsilon_n)_n$
defines a $\D$-operator for the  Charlier polynomials $(c_{n}^{a})_n$ and that $\D=\nabla$.

The eigenvalues of $(c_{n}^{a})_n$ with respect to $D_{a}$ are $\theta_n=-n$, hence
the definition of $(\gamma _n)_n$, $(\lambda _n)_n$  and the relationships between the polynomials $P_1$ and $P_2$
automatically give
$$
\lambda _n-\lambda_{n-1}=\gamma_n.
$$
The Lemma \ref{fl1v} now shows that the polynomials $q_{n}$, $n\ge 0$, defined by (\ref{qnno}) are eigenfunctions of the difference operator
$\displaystyle D=P_1(D_{a})+\nabla P_2(D_{a})$
with eigenvalues $\lambda_n$.

Since $P_1$ and $P_2$ are polynomials of degrees $k+1$ and $k$, respectively, and $D_a$ has genre $(-1,1)$,
the operators $D(P_1)$ and $D(P_2)$ have order $2k+2$ and $2k$ and genre $(-k-1,k+1)$ and $(-k,k)$, respectively.
Write now $f_{-1}=x$ and $f_1=a$ for the coefficients of $\Sh_{-1}$ and $\Sh_1$ in $D_a$, and $u_1$, $u_2$ for the leading coefficients
of the polynomials $P_1$ and $P_2$, respectively. A straightforward computation
gives that the coefficients of $\Sh_{-k-1}$ and $\Sh_{k+1}$ in the operator $D$ are $f_{-1}(x-1)\cdots f_{-1}(x-k)(u_1f_{-1}(x)-u_2)$
and $u_1f_{1}(x)\cdots f_{1}(x+k)$, respectively. We can see that both are different to zero. This shows that
the operator $D$ has order $2k+2$ and genre $(-k-1,k+1)$.

\end{proof}

We next find some particular choices of $P_2$ for which the sequence of polynomials $(q_n)_n$ (\ref{qnno}) is orthogonal with respect to a
measure.

For $k$ a positive integer, let $\tilde \rho_{k,a}$ be the measure defined by ($a\not =0$)
\begin{align}\label{tchw}
\tilde \rho_{k,a}&=(x+1)\cdots (x+k)\rho_a(x+k+1)\\\nonumber
&=(-1)^kk!\delta_{-k-1}+\sum _{x=0}^\infty \frac{a^{x+k+1}}{(x+k+1)x!}\delta _x.
\end{align}
To simplify the notation we will sometimes remove the dependence of $a$ and $k$ and write $\tilde \rho=\tilde \rho_{k,a}$.

In order to find the orthogonal polynomials with respect to $\tilde \rho _{k,a}$ we need the following Lemma.

\begin{lemma}\label{chxx} For a positive integer $k$ and a non-null real number $a$, we have
\begin{align}\label{chx+k+1}
\langle \tilde \rho_{k,a},c_{n}^{a}\rangle = (-1)^ne^ak!c_{k}^{-a}(-n-1).
\end{align}
\end{lemma}

\begin{proof}
Write
\begin{align*}
\xi_n&=\langle \tilde \rho_{k,a},c_{n}^{a}\rangle, \quad n\ge 0,\\
\zeta _{n}&=(-1)^nk!e^ac_{k}^{-a}(-n-1),\quad n\ge 0.
\end{align*}
The three term recurrence relation (\ref{repnch}) for $(c_{n}^a)_n$ gives that
\begin{align}\label{mecc1a}
\xi_{1}+(a+k+1)\xi_{0}-a^{k+1}e^a&=0,\\
\label{mecc1}
(n+1)\xi_{n+1}+(n+a+k+1)\xi_{n}+a\xi_{n-1}&=0, \quad n\ge 1.
\end{align}
We now prove that also
\begin{align}\label{mecc2a}
\zeta_{1}+(a+k+1)\zeta_{0}-a^{k+1}e^a&=0,\\ \label{mecc2}
(n+1)\zeta_{n+1}+(n+a+k+1)\zeta_{n}+a\zeta_{n-1}&=0, \quad n\ge 1.
\end{align}
Indeed, for $n\ge 1$ we have
\begin{align*}
&(n+1)\zeta_{n+1}+(a+n+k+1)\zeta_{n}+a\zeta_{n-1}=(-1)^{n+1}k!e^a\\
&\hspace{1.5cm}\times \left( (n+1)c_{k}^{-a}(-n-2)-(n+a+k+1)c_{k}^{-a}(-n-1)+ac_{k}^{-a}(-n)\right).
\end{align*}
Then, by writing $x=-n-1$ in the second order difference equation (\ref{sodech}) for the Charlier polynomials $(c_{k}^{-a})_k$, we find
$$
(n+1)\zeta_{n+1}+(a+n+k+1)\zeta_{n}+a\zeta_{n-1}=0.
$$
Proceeding in a similar way for $n=0$, and using that $c_k^{-a}(0)=a^k$, we get (\ref{mecc2a}).

Since the sequences $(\xi_n)_n$ and $(\zeta_n)_n$ satisfy the same recurrence relation,
it is enough to prove
that $\xi_0=\zeta_0$. Write then $\eta_0=\langle \rho_a(x+1),1\rangle =e^a$, and
$$
\eta _k=\langle \rho_{a}(x+k+1),(x+1)\cdots (x+k)\rangle , \quad k\ge 1.
$$
From  the definition of $\tilde \rho_{k,a}$ we have $\eta_k=\xi_0$.
Using that $x\rho_a(x)=a\rho_a(x-1)$, we get
\begin{align*}
\eta _k&=\langle \rho_{a}(x+k+1),(x+1)\cdots (x+k)\rangle =\langle \rho_{a}(x+k),x(x+1)\cdots (x+k-1)\rangle \\
&=\langle a\rho_{a}(x+k-1),(x+1)\cdots (x+k-1)\rangle-k\langle \rho_{a}(x+k),(x+1)\cdots (x+k-1)\rangle\\
&=a^{k}\langle \rho_{a},1\rangle -k\eta_{k-1}=a^{k}e^a -k\eta_{k-1}.
\end{align*}
This means that the sequence $\eta _k$, $k\ge 1$, is characterized by the recurrence relation
$$
\eta_k+k\eta_{k-1}=a^{k}e^a, \quad k\ge 1,
$$
with initial condition $\eta _0=e^a$.
The proof will finish if we prove that also the sequence $(e^ak!c_{k}^{-a}(-1))_k$  satisfies the same recursion.
But, by applying $k-1$ times the three term recurrence relation for the Charlier polynomials $(c_{k}^{-a})_k$ (see (\ref{repnch})), we get
$$
c_{k}^{-a}(-1)+c_{k-1}^{-a}(-1)=\frac{a}{k}(c_{k-1}^{-a}(-1)+c_{k-2}^{-a}(-1))
=\frac{a^{k-1}}{k!}(c_{1}^{-a}(-1)+c_{0}^{-a}(-1))=\frac{a^{k}}{k!}.
$$
\end{proof}

We are now ready to prove that the orthogonal polynomials with respect to $\tilde \rho _{k,a}$ (\ref{tchw}) are a particular case
of Theorem \ref{th5.1ch}.

\begin{theorem} \label{lm51}
For $k\ge 1$, let $a$ be a non-null real number satisfying that $c_{k}^{-a}(-n)\not =0$, $n\ge 1$, where $c_{k}^{-a}$, $k\ge 1$, are
Charlier polynomials (see (\ref{defchp})).
We define the sequences of numbers $(\lambda_n)_n$, $(\gamma _n)_{n\ge 1}$ and $(\beta _n)_{n\ge 1}$  by
\begin{align}\label{eigchk}
\lambda _n&=-c_{k+1}^{-a}(-n),\\\label{defdelch}
\gamma_n&=c_{k}^{-a}(-n),\\
\label{defbetch}
\beta_n&=\frac{\gamma_{n+1}}{\gamma_n},\quad n\ge 1,
\end{align}
and the sequence of polynomials $(q_n)_n$ by $q_0=1$, and
\begin{equation}\label{defqch}
q_n(x)=c_{n}^{a}(x)+\beta _nc_{n-1}^{a}(x), \quad n\ge 1.
\end{equation}
Then the polynomials $(q_n)_n$ are orthogonal with respect
to the measure $\tilde \rho $ (\ref{tchw}).
Moreover, consider the difference operator of order $2k+2$ and genre $(-k-1,k+1)$ defined by
$$
D=P_1(D_{a})+\nabla P_2(D_{a}),
$$
where $D_{a}$ is the second order difference operator for the Charlier polynomials (\ref{sodech})
and  $P_1$, $P_2$ are the polynomials of degrees $k+1$ and $k$, respectively, defined by
\begin{align}\label{defp1ch}
P_1(x)&=-c_{k+1}^{-a}(x),\\\label{defp2ch}
P_2(x)&=c_{k}^{-a}(x-1).
\end{align}
Then $D(q_n)=\lambda_nq_n$.
\end{theorem}

\begin{proof}
In the notation of Lemma \ref{addel}, we have $\lambda =-k-1$, $\nu=\tilde \rho _{k,a}$, $\mu =a^{k+1}\rho _a$, $p_n=c_n^a$  and $M=0$.
Lemma \ref{chxx} gives that $\alpha_n=\langle \nu,p_n\rangle= (-1)^ne^ak!c_{k}^{-a}(-n-1)$. The orthogonality of the polynomials $(q_n)_n$ with
respect to $\tilde \rho _{k,a}$ is now a consequence of Lemma \ref{addel} and (\ref{defdelch}), (\ref{defbetch}).

The second part of the Theorem is a straightforward consequence of the Theorem \ref{th5.1ch}. We have just to check that
$P_2(x)=P_1(x-1)-P_1(x)$:
$$
P_1(x-1)-P_1(x)=c_{k+1}^{-a}(x)-c_{k+1}^{-a}(x-1)=\Delta (c_{k+1}^{-a}(x-1))
=c_{k}^{-a}(x-1)=P_2(x).
$$
\end{proof}

Except for an appropriate choice of the polynomial $P_2$ (as before), the polynomials $(q_n)_n$ (\ref{qnno}) in Theorem \ref{th5.1ch}
seem not to be orthogonal
with respect to a measure. Anyway, in a subsequent paper \cite{du3}, we prove that they enjoy certain orthogonality property which implies that
they satisfy a higher order recurrence relation of the form
$$
(x+1)(x+2)\cdots (x+k_0+1)q_n(x)=\sum_{j=-k_0-1}^{k_0+1}a_{n,j}q_{n+j}(x),
$$
where $k_0\le\deg(P_2)$.

\bigskip

Conjecture 2 in the Introduction for the Charlier weight $\rho _a$
and the finite set $F=\{ 1, 2, \cdots, k\}$ can be deduced as a Corollary of the previous Theorem.

\begin{corollary}\label{jodch}
For a positive integer $k$, write $F=\{ 1, 2, \cdots , k\}$ and $\pp(x)=\prod _{j\in F}(x-j).$ Assume that $a\not =0$ and
that $c_{k}^{-a}(-n)\not =0$, $n\ge 1$. Then,
the orthogonal polynomials with respect to the measure $\pp\rho _a$ are common eigenfunctions of a $(2k+2)$-order difference operator
of genre $(-k-1,k+1)$,
where the measure $\rho_a$ is the Charlier measure (\ref{chw}).
\end{corollary}

\begin{proof}
The proof is an easy consequence of the previous Theorem, taking into account that the measure $\pp\rho_a$ is
equal to the measure $\tilde \rho$ shifted by $k+1$:
$\pp\rho_a=\tilde \rho (x-k-1)$.
\end{proof}

\begin{remark}\label{remch}
The most interesting case in the previous Theorems is when the new measures  $\tilde\rho_{k,a}$ (\ref{tchw}) in Theorem \ref{lm51} and
$\pp \rho_{a}$ in Corollary \ref{jodch} are positive. This obviously happens
when $a>0$ and $k$ is even.
In that case, the hypothesis $c_{k}^{-a}(-n)\not =0$,  $n\ge 1$, in Theorem \ref{lm51} and Corollary
\ref{jodch} always fulfil. Indeed, by definition of  the measure $\tilde \rho_{k,a}$ (\ref{tchw}) we have that
$$
\tilde \rho_{k,a} =(x+1)\cdots(x+k)\rho_a(x+k+1)
$$
where $\rho_a(x+k+1)$ is the Charlier weight (\ref{chw}) shifted by $-k-1$.

That means that the measure $\tilde \rho_{k,a}$ is a Christoffel transform of the measure $\rho_a(x+k+1)$.
Write $\Lambda _n$, $n\ge 0$, for the matrices
\begin{equation}\label{condepch}
\Lambda _n= (c_{n+j-1}^{a}(i))_{i,j=1,\cdots , k}.
\end{equation}
Hence, according to (\ref{mata0}) in the Preliminaries,
the measure $\tilde \rho_{k,a}$ has a sequence $(q_n)_{n}$ of orthogonal polynomials if and only if $\det \Lambda_n\not =0$, $n\ge 0$.
Now for $a>0$ and $k$ even, the measure $\tilde \rho_{k,a}$ is positive with infinitely many points in its support. So it always has
a sequence of orthogonal polynomials,
and then $\det \Lambda_n\not =0$, $n\ge 0$. In \cite{duc} we have found a very nice closed-form expression for the determinant of the
matrices $\Lambda_n$ (\ref{condepch}):
\begin{equation}\label{chconjdet}
\det(\Lambda_n)=(-1)^{kn}a^{(n-1)k}\left( \prod_{j=1}^{k}\frac{j!}{(n+j-1)!}\right) c_{k}^{-a}(-n).
\end{equation}

We can then conclude that the hypothesis
$c_{k}^{-a}(-n)\not =0$, $n\ge 1$, in Theorem \ref{lm51}  and Corollary
\ref{jodch} always fulfil for $a>0$ and $k$ even.

There is a similar formula to (\ref{chconjdet}) for each of the examples studied in Sections 5, 6 and 7 of this paper. Actually,
they seem to be particular cases of a surprising symmetry for certain Casorati determinants associated to
the classical discrete orthogonal polynomials (see \cite{duc}).
\end{remark}

\section{Meixner case}
For $a\not =0, 1$ we write $(m_{n}^{a,c})_n$ for the sequence of Meixner polynomials defined by
\begin{equation}\label{defmep}
m_{n}^{a,c}(x)=(-1)^n\sum _{j=0}^n a^{-j}\binom{x}{j}\binom{-x-c}{n-j}
\end{equation}
(we have taken a slightly different normalization from the one used in \cite{Ch}, pp. 175-7, from where
the next formulas can be easily derived; see also \cite{KLS}, pp, 234-7 or \cite{NSU}, ch. 2).
Meixner polynomials are eigenfunctions of the following second order difference operator
\begin{equation}\label{sodeme}
D_{a,c} =x\Sh_{-1}-[(1+a)x+ac]\Sh_0+a(x+c)\Sh_1,\qquad D_{a,c} (m_{n}^{a,c})=n(a-1)m_{n}^{a,c},\quad n\ge 0.
\end{equation}
When $a\not =0, 1$, they satisfy the following three term recurrence formula ($m_{-1}=0$)
\begin{equation}\label{trme}
xm_n=\frac{a(n+1)}{a-1}m_{n+1}-\frac{(a+1)n+ac}{a-1}m_n+\frac{n+c-1}{a-1}m_{n-1}, \quad n\ge 0
\end{equation}
(to simplify the notation we remove the parameters in some formulas).
Hence, for $a\not =0,1$ and $c\not =0,-1,-2,\cdots $, they are always orthogonal with respect to a moment functional $\rho_{a,c}$, which we
normalize
by taking $\langle \rho_{a,c},1\rangle =\Gamma(c)$. For $0<\vert a\vert<1$ and $c\not =0,-1,-2,\cdots $, we have
\begin{equation}\label{mew}
\rho_{a,c}=(1-a)^c\sum _{x=0}^\infty \frac{a^{x}\Gamma(x+c)}{x!}\delta _x.
\end{equation}
The moment functional $\rho_{a,c}$ can be represented by a positive measure only when $0<a<1$ and $c>0$.

Meixner polynomials satisfy the following identities
\begin{align}\label{sdm}
\Delta (m_{n}^{a,c})&=\frac{a-1}{a}m_{n-1}^{a,c+1},\\ \label{sdm2}
m_{n}^{a,c+1}(x-1)&=\sum _{j=0}^na^{-j}m_{n-j}^{a,c}(x),\\
\label{sdm2b}
m_{n}^{a,c+1}(x)&=\sum _{j=0}^nm_{n-j}^{a,c}(x),\\
\label{sdm3}
\frac{1}{a}m_{n-1}^{a,c}(x)&=m_{n}^{a,c}(x)-m_{n}^{a,c-1}(x+1).
\end{align}

\bigskip
For Meixner polynomials we have found two different $\D$-operators.

\begin{lemma}\label{lTme}
For $a\not =0,1$, consider the Meixner polynomials $(m_{n}^{a,c})_n$ (\ref{defmep}). Then, the operators $\D _i$,
$i=1,2$, defined by (\ref{defTo}) from the sequences ($n\ge 0$)
\begin{align}\label{veme1}
\varepsilon _{n,1}&=-1,\\\label{veme2}
\varepsilon _{n,2}&=-1/a,
\end{align}
are $\D$-operators for $(m_{n}^{a,c})_n$ and the algebra $\A$  (\ref{algdiff}) of difference
operators  with polynomial coefficients.
More precisely
\begin{align}\label{dome1}
\displaystyle \D_1&=\frac{a}{1-a}\Delta ,\\\label{dome2}
\displaystyle \D_2&=\frac{1}{1-a}\nabla .
\end{align}
\end{lemma}

\begin{proof}
Indeed, the sequence (\ref{veme1}) defines the $\D$-operator (\ref{defTo})
$$
\D_1 (m_{n}^{a,c})=\sum_{j=1}^\infty (-1)^{j+1}\zeta ^j(m_{n}^{a,c})=\sum_{j=1}^n (-1)^{j+1}(-1)^jm_{n-j}^{a,c}
=m_{n}^{a,c}(x)-\sum_{j=0}^n m_{n-j}^{a,c}.
$$
Using (\ref{sdm2b}), (\ref{sdm3}) and (\ref{sdm}), we get
\begin{align*}
\D_1 (m_{n}^{a,c})&=m_{n}^{a,c}(x)-m_{n}^{a,c+1}(x)=m_{n}^{a,c}(x)-\frac{1}{a}m_{n-1}^{a,c+1}(x)
-m_{n}^{a,c}(x+1)\\& =\frac{a}{1-a}\Delta(m_{n}^{a,c}).
\end{align*}
The proof for $\D_2$ is similar and it is omitted.

\end{proof}

Each one of the $\D$-operators displayed in the previous Lemma together with Lemma \ref{fl1v} gives rise to the corresponding
class of polynomials satisfying higher order difference equations.

\begin{theorem}\label{th5.1me1} Let $P_1$ be a polynomial of degree $k+1$, $k\ge 1$,
and write $P_2(x)=P_1(x+a-1)-P_1(x)$ (so that $P_2$ is a polynomial of degree $k$). We assume that $P_2(n(a-1))\not =0$, $n\ge 0$, and
define the sequences of numbers
\begin{align}\nonumber
\gamma_{n+1}&=P_2(n(a-1)),\quad n\ge 0,\\\nonumber
\lambda_n&=P_1(n(a-1)),\quad n\ge 0,\\\label{beme}
\beta_{n,1}&=-\frac{\gamma_{n+1}}{\gamma_n}\quad n\ge 1,\\
\beta_{n,2}&=-\frac{1}{a}\frac{\gamma_{n+1}}{\gamma_n}\quad n\ge 1,
\end{align}
and the sequence of polynomials $q_{0,i}=1$, and for $n\ge 1$
\begin{equation}\label{qnme}
q_{n,i}(x)=m_{n}^{a,c}(x)+\beta_{n,i}m_{n-1}^{a,c}(x),\quad i=1,2,
\end{equation}
where $m_{n}^{a,c}$ is the Meixner polynomial defined by (\ref{defmep}) ($a\not =0, 1$).
Consider the second order difference operator $D_{a,c}$ (\ref{sodeme}) with respect to which
the Meixner polynomials are eigenfunctions. Write finally $D_i$, $i=1,2$, for the difference operators of order $2k+2$ and genre $(-k-1,k+1)$
\begin{align*}
D_1&=P_1(D_{a,c})+\frac{a}{1-a}\Delta P_2(D_{a,c}),\\
D_2&=P_1(D_{a,c})+\frac{1}{1-a}\nabla P_2(D_{a,c}).
\end{align*}
Then $D_i(q_{n,i})=\lambda_nq_{n,i}$, $n\ge 0$, $i=1,2$.
\end{theorem}

\begin{proof}

The theorem can be proved as Theorem \ref{th5.1ch} using Lemma \ref{fl1v} and the $\D$-operators in Lemma \ref{lTme}.

\end{proof}

\subsection{Meixner I}
For a convenient choice of the polynomial $P_1$, the polynomials $(q_{n,1})_n$ in the previous Theorem turn out to be orthogonal with respect
to a moment functional. Indeed, for $k$ a positive integer, $a\not =0,1$ and $c\not =k+1, k, ,k-1, \cdots$,
we consider the moment functional
\begin{equation}\label{tmew2}
\tilde \rho_{k,a,c}=(x+c-1)\cdots (x+c-k)\rho_{a,c-k-1}.
\end{equation}
For $0<\vert a\vert <1$ and $c\not =k+1,k,\cdots $ the moment funcional
$\tilde \rho_{k,a,c}$ can be represented by the measure
\begin{equation}\label{tmewb2}
\tilde \rho_{k,a,c}=(1-a)^{c-k-1}\sum _{x=0}^\infty \frac{a^{x}\Gamma(x+c)}{(x+c-k-1)x!}\delta _x.
\end{equation}
To simplify the notation we will sometimes remove the dependence of $a,c$ and $k$ and write $\tilde \rho =\tilde \rho_{k,a,c}.$

The following Lemma will be needed to find the orthogonal polynomials with respect to $\tilde \rho_{k,a,c}$.

\begin{lemma}\label{lme1x} For $a\not =0,1$, $c\not =k+1,k,\cdots$ and $n\ge 0$, we have
\begin{equation}\label{me1x+k+1}
\langle \tilde \rho_{k,a,c},m_{n}^{a,c}\rangle =\frac{(-1)^kk!\Gamma (c-k-1)m_{k}^{1/a,-c+2}(-n-1)}{(1-a)^{k}}.
\end{equation}
\end{lemma}

\begin{proof}
We proceed in an analogous way to Lemma \ref{chxx}.
Write
\begin{align*}
\xi_n&=\langle \tilde \rho_{k,a,c},m_{n}^{a,c}\rangle , \quad n\ge 0, \quad \xi_{-1}=\frac{\Gamma(c-1)}{(1-a)^k},\\
\zeta_n&=\frac{(-1)^kk!\Gamma (c-k-1)m_{k}^{1/a,-c+2}(-n-1)}{(1-a)^{k}},\quad n\ge -1.
\end{align*}
The three term recurrence relation (\ref{trme}) for $(m_{n}^{a,c})_n$ gives that
\begin{equation}\label{mecme1}
a_n\xi_{n+1}+(b_n+c-k-1)\xi_{n}+c_n\xi_{n-1}=0, \quad n\ge 1,
\end{equation}
where $a_n$, $b_n$ and $c_n$ are the recurrence coefficients of the Meixner polynomials $(m_{n}^{a,c})_n$.
Using that $(x+c)\rho_{a,c}=\rho_{a,c+1}/(1-a)$, the normalization for the Meixner weight and the value of $\xi_{-1}$,
we also have (\ref{mecme1}) for $n=0$.

We now prove that also
\begin{equation}\label{mecme2}
a_n\zeta_{n+1}+(b_n+c-k-1)\zeta_{n}+c_n\zeta_{n-1}=0, \quad n\ge 0.
\end{equation}
Taking into account the definition of $\zeta _n$, this is equivalent to
\begin{align*}
a(n+1)m_{k}^{1/a,-c+2}&(-n-2)-((a+1)(n+1)-k(1-a)+c-2))m_{k}^{1/a,-c+2}(-n-1)\\
&\quad\quad \quad +(n+c-1)m_{k}^{1/a,-c+2}(-n)=0, \quad n\ge 0.
\end{align*}
But this follows straightforwardly by writing  $x=-n-1$ in the second order difference equation for the Meixner polynomials
$(m_{k}^{1/a,-c+2})_k$ (see (\ref{sodeme})).

Since the sequences $(\xi_n)_n$ and $(\zeta_n)_n$ satisfy the same recurrence relation ((\ref{mecme1}), (\ref{mecme2})), it is enough to prove
that $\xi_{-1}=\zeta_{-1}$ and $\xi_0=\zeta_0$.

A straightforward computation from (\ref{defmep}) shows that $\xi_{-1}=\zeta_{-1}$.

In order to prove that also $\xi_0=\zeta_0$, we proceed as follows.
Write
$$
\eta _{k,c}=\langle \rho_{a,c-k-1},(x+c-1)\cdots (x+c-k)\rangle.
$$
From  the definition of $\tilde \rho_{k,a,c}$ we have $\eta_{k,c}=\xi_0$.
Using that $(x+c)\rho_{a,c}=\rho_{a,c+1}/(1-a)$, we get
\begin{align*}
\eta _{k,c}&=\langle \rho_{a,c-k-1},(x+c-1)\cdots (x+c-k)\rangle \\
&= \frac{\langle\rho_{a,c-k},(x+c-2)\cdots (x+c-k)\rangle}{1-a}+k\langle \rho_{a,c-k-1},(x+c-2)\cdots (x+c-k)\rangle\\
&=\frac{\langle \rho_{a,c-1},1\rangle}{(1-a)^k} +k\eta_{k-1,c-1}=k\eta_{k-1,c-1}+\frac{\Gamma(c-1)}{(1-a)^{k}}.
\end{align*}
On the other hand, if we write
$$
\tau _{k,c}=\zeta_0=\frac{(-1)^kk!\Gamma (c-k-1)m_{k}^{1/a,-c+2}(-1)}{(1-a)^{k}},
$$
we have using (\ref{sdm}) and (\ref{defmep})
\begin{align*}
\tau_{k,c}-k\tau_{k-1,c-1}&=\frac{(-1)^kk!\Gamma(c-k-1)}{(1-a)^k}(m_{k}^{1/a,-c+2}(-1)+(1-a)m_{k-1}^{1/a,-c+3}(-1))\\
&=\frac{(-1)^kk!\Gamma(c-k-1)}{(1-a)^k}m_{k}^{1/a,-c+2}(0)=\frac{\Gamma(c-1)}{(1-a)^k}.
\end{align*}
Proceeding by induction on $k$ it is now easy to prove that $\eta_{k,c}=\tau_{k,c}$, and then $\xi_0=\zeta_0$.

\end{proof}

We are now ready to prove that the orthogonal polynomials with respect to $\tilde \rho _{k,a,c}$ (\ref{tmew2}) are a particular case
of Theorem \ref{th5.1me1}.

\begin{theorem} \label{lm62}
For $k\ge 1$, let $a$ and $c$ be real numbers satisfying that $a\not =0,1$, $c\not =k+1, k, k-1,\cdots $
and $m_{k}^{1/a,-c+2}(-n)\not =0$, $n\ge 1$, where $m_{k}^{1/a,-c+2}$, $k\ge 1$, are Meixner polynomials (see (\ref{defmep})).
We define the sequences of numbers $(\lambda _n)_n$, $(\gamma _n)_{n\ge 1}$ and $(\beta _n)_{n\ge 1}$ by
\begin{align}\nonumber
\lambda _n&=\frac{1}{a-1}m_{k+1}^{1/a,-c+1}(-n),\\
\label{defdelme2}
\gamma_n&=m_{k}^{1/a,-c+2}(-n),\\
\label{defbetme2}
\beta_n&=-\frac{\gamma_{n+1}}{\gamma_n},
\end{align}
and the sequence of polynomials $(q_n)_n$ by $q_0=1$, and
\begin{equation}\label{defqme2}
q_n(x)=m_{n}^{a,c}(x)+\beta _nm_{n-1}^{a,c}(x), \quad n\ge 1.
\end{equation}
Then the polynomials $(q_n)_n$ are orthogonal with respect
to the moment functional $\tilde \rho_{k,a,c} $ (\ref{tmew2}).
Moreover, consider the difference operator of order $2k+2$ and genre $(-k-1,k+1)$ defined by
$$
D=P_1(D_{a,c})+\frac{a}{1-a}\nabla P_2(D_{a,c}),
$$
where $D_{a,c}$ is the second order difference operator for the Meixner polynomials (\ref{sodeme})
and $P_1$ and $P_2$ are the polynomials of degrees $k+1$ and $k$, respectively, defined by
\begin{align*}
P_1(x)&=\frac{1}{a-1}m_{k+1}^{1/a,-c+1}\left( -\frac{x}{a-1}\right) ,\\
P_2(x)&=m_{k}^{1/a,-c+2}\left( -\frac{x}{a-1}-1\right) .
\end{align*}
Then $D(q_n)=\lambda_nq_n$.
\end{theorem}

\begin{proof}
The Theorem can be proved in an analogous way to Theorem \ref{lm51}.
\end{proof}

Conjecture 1 in the Introduction for the Meixner weight $\rho _{a,c}$
and the finite set $F=\{ 1, 2, \cdots, k\}$ can be deduced as a Corollary of the previous Theorem:

\begin{corollary} \label{lm623} For a positive integer $k$, write $F=\{ 1, 2, \cdots , k\}$ and $\ps(x)=\prod _{j\in F}(x+c+j).$
Assume that $a\not =0,1$,
$c\not =0,-1,-2, \cdots$ and that
$$
m_{k}^{1/a,-c-k+1}(-n)\not =0, \quad n\ge 1.
$$
Then
the orthogonal polynomials with respect to the measure $\ps\rho _{a,c}$ are common eigenfunctions of a $(2k+2)$-order difference operator
of genre $(-k-1,k+1)$,
where $\rho_{a,c}$ is the moment functional for the Meixner polynomials.
\end{corollary}

\begin{proof}
The proof is an easy consequence of the previous Theorem, taking into account that the measure $\ps\rho_{a,c}$
is equal to the measure $\tilde \rho _{k,a,c+k+1}$.
\end{proof}

\subsection{Meixner II}
There is other choice of the polynomial $P_1$ for which the polynomials $(q_{n,2})_n$ in Theorem \ref{th5.1me1} are also orthogonal
with respect to a moment functional. Indeed, for $k$ a positive integer, $a\not =0,1$ and $c\not =k+1,k,\cdots $ let $\tilde \rho_{k,a,c}$ be the moment
functional defined by
\begin{equation}\label{tmew}
\tilde \rho_{k,a,c}=(x+1)\cdots (x+k)\rho_{a,c-k-1}(x+k+1).
\end{equation}
In particular, for $0<\vert a\vert <1$ and $c\not = k+1,k,\cdots $ the moment funcional
$\tilde \rho_{k,a,c}$ can be represented by the measure
\begin{equation}\label{tmewb}
\tilde \rho_{k,a,c}=(1-a)^{c-k-1}\left((-1)^kk!\Gamma(c-k-1)\delta_{-k-1}+\sum _{x=0}^\infty \frac{a^{x+k+1}\Gamma(x+c)}{(x+k+1)x!}\delta _x \right) .
\end{equation}
To simplify the notation we remove the dependence of $a,c$ and $k$ and write $\tilde \rho =\tilde \rho_{k,a,c}$
(the proofs of the following theorems are similar to the ones in the previous Sections and are omitted).

\begin{lemma} For $a\not =0,1$, $c\not =k+1,k,\cdots$ and $n\ge 0$, we have
$$
\langle \tilde \rho_{k,a,c},m_{n}^{a,c}\rangle =\frac{(-1)^kk!\Gamma (c-k-1)m_{k}^{a,-c+2}(-n-1)}{a^{n-k}(1-a)^{k}}.
$$
\end{lemma}

\begin{theorem} \label{lm61}
For $k\ge 1$, let $a$ and $c$ be real numbers satisfying that $a\not =0,1$, $c\not =k+1, k, ,k-1, \cdots $
and $m_{k}^{a,-c+2}(-n)\not =0$, $n\ge 1$, where $m_{k}^{a,-c+2}$, $k\ge 1$, are  Meixner polynomials (see (\ref{defmep})).
We define the sequences of numbers $(\lambda _n)_n$, $(\gamma _n)_{n\ge 1}$ and $(\beta _n)_{n\ge 1}$ by
\begin{align}\label{eigmek}
\lambda _n&=-\frac{a}{a-1}m_{k+1}^{a,-c+1}(-n),\\\label{defdelme}
\gamma_n&=m_{k}^{a,-c+2}(-n),\\
\label{defbetme}
\beta_n&=-\frac{1}{a}\frac{\gamma_{n+1}}{\gamma_n},
\end{align}
and the sequence of polynomials $(q_n)_n$ by $q_0=1$, and
\begin{equation}\label{defqme}
q_n(x)=m_{n}^{a,c}(x)+\beta _nm_{n-1}^{a,c}(x), \quad n\ge 1.
\end{equation}
Then the polynomials $(q_n)_n$ are orthogonal with respect
to the moment functional $\tilde \rho $ (\ref{tmew}). Moreover,
consider the difference operator of order $2k+2$ and genre $(-k-1,k+1)$ defined by
$$
D=P_1(D_{a,c})+\frac{1}{1-a}\nabla P_2(D_{a,c}),
$$
where $D_{a,c}$ is the second order difference operator for the Meixner polynomials (\ref{sodeme})
and  $P_1$ and $P_2$ are the polynomials of degrees $k+1$ and $k$, respectively, defined by
\begin{align}\label{defp1me}
P_1(x)&=-\frac{a}{a-1}m_{k+1}^{a,-c+1}\left( -\frac{x}{a-1}\right) ,\\\label{defp2me}
P_2(x)&=m_{k}^{a,-c+2}\left( -\frac{x}{a-1}-1\right) .
\end{align}
Then $D(q_n)=\lambda_nq_n$.
\end{theorem}

Conjecture 2 in the Introduction for the Meixner weight $\rho _{a,c}$
and the finite set $F=\{ 1, 2, \cdots, k\}$ can be deduced as a Corollary of the previous Theorem:

\begin{corollary}\label{lm613} For a positive integer $k$, write $F=\{ 1, 2, \cdots , k\}$ and $\pp(x)=\prod _{j\in F}(x-j)$. Assume
that $a\not =0,1$, $c\not =0, -1, -2, \cdots$ and that
$$
m_{k}^{a,-c-k+1}(-n)\not =0, \quad n\ge 1.
$$
Then, the orthogonal polynomials with respect to the measure $\pp\rho _{a,c}$
are common eigenfunctions of a $(2k+2)$-order difference operator of genre $(-k-1,k+1)$,
where $\rho_{a,c}$ is the moment functional for the Meixner polynomials $(m_n^{a,c})_n$.
\end{corollary}

\begin{remark}
Let us note that in some cases the measure $\pp \rho_{a,c}$ can be positive even though the measure $\rho _{a,c}$ is not.
Indeed, that it the case, for instance, when $k$ is odd, $0<a<1$ and $-2j-1<c<-2j$, $j=0,\cdots , (k-1)/2$, or
when $k$ is even, $0<a<1$ and $-2j<c<-2j-1$, $j=1,\cdots , k/2$.

\end{remark}

\section{Krawtchouk case}
For $a\not=0,-1$, we write $(k_{n}^{a,N})_n$ for the sequence of  Krawtchouk polynomials defined by
\begin{equation}\label{defkrp}
k_{n}^{a,N}(x)=\frac{1}{n!}\sum_{j=0}^n(-1)^{n+j}\frac{(1+a)^{j-n}}{a^{j-n}}\frac{(-n)_j(-x)_j(N-n)_{n-j}}{j!}.
\end{equation}
Krawtchouk polynomials are eigenfunctions of the following second order difference operator ($n\ge 0$)
\begin{equation}\label{sodekr}
D_{a,N} =x\Sh_{-1}-(x-a(x-N+1))\Sh_0-a(x-N+1)\Sh_1,\qquad D_{a,N} (k_{n}^{a,N})=-n(1+a)k_{n}^{a,N}.
\end{equation}
For $a\not=0,-1$ and $N\not=1, 2, \cdots$, they are always orthogonal with respect to a moment functional $\rho _{a,N}$,
which we normalize by taking
$\langle \rho_{a,N},1\rangle=1$.
When $N$ is a positive integer and $a>0$, the first $N$ polynomials are orthogonal with respect to the positive Krawtchouk measure
\begin{equation}\label{krw}
\rho_{a,N}=\frac{\Gamma(N)}{(1+a)^{N-1}}\sum _{x=0}^{N-1} \frac{a^{x}}{\Gamma(N-x)x!}\delta _x.
\end{equation}
The structural formulas for $(k_{n}^{a,N})_n$ can be found in \cite{NSU}, pp. 30-53 (see
also \cite{KLS}, pp, 204-211).

Proceeding as in the previous Sections, we have identified two $\D$-operators for Krawtchouk polynomials (the proof is omitted).

\begin{lemma}\label{lTkr}
For $a\not =0,-1$ and $N\in \RR$, consider the Krawtchouk polynomials $(k_{n}^{a,N})_n$ (\ref{defkrp}). Then
the operators $\D _i$, $i=1,2$, defined by (\ref{defTo}) from the sequences ($n\ge 0$)
\begin{align*}
\varepsilon _{n,1}&=1/(1+a),\\
\varepsilon _{n,2}&=-a/(1+a),
\end{align*}
are $\D$-operators for $(k_{n}^{a,N})_n$ and the algebra $\A$  (\ref{algdiff}) of difference
operators  with polynomial coefficients.
More precisely
\begin{align*}
\displaystyle \D_1&=\frac{1}{1+a}\nabla ,\\
\displaystyle \D_2&=-\frac{a}{1+a}\Delta .
\end{align*}
\end{lemma}

Each one of the $\D$-operators displayed in the previous Lemma together with Lemma \ref{fl1v} gives rise to the corresponding
class of polynomials satisfying higher order difference equations.
However, we only proceed with the first  $\D$-operator because due to the symmetries of Krawtchouk polynomials,
the examples generated from the second $\D$-operator are essentially the same: if we write $q^{(1)}_{n,a,N}$, $q^{(2)}_{n,a,N}$,
respectively, for
the polynomials defined by (\ref{qnkr}) below from the first and the second $\D$-operators above, it is not difficult to see (using
that $k_{n}^{a,N}(x)=(-1)^nk_{n}^{1/a,N}(N-1-x)$)
that
$$
q^{(1)}_{n,a,N}(x)=(-1)^nq^{(2)}_{n,1/a,N}(-x+N-1)
$$
(using the first $\D$-operator we will prove Conjecture 2 in the Introduction for the Krawtchouk polynomials, while using the second $\D$-operator one can
similarly prove the equivalent Conjecture 2' in \cite{du1} for the Krawtchouk polynomials).

\begin{theorem}\label{th5.1kr} Let $P_1$ be a polynomial of degree $k+1$, $k\ge 1$,
and write $P_2(x)=P_1(x-1-a)-P_1(x)$ (so that $P_2$ is a polynomial of degree $k$). We assume that $P_2(-n(1+a))\not =0$, $n\ge 0$, and
define the sequences of numbers
\begin{align}\nonumber
\gamma_{n+1}&=P_2(-n(1+a)),\quad n\ge 0,\\\nonumber
\lambda_n&=P_1(-n(1+a)),\quad n\ge 0,\\\label{bekr}
\beta_n&=\frac{1}{1+a}\frac{\gamma_{n+1}}{\gamma_n}\quad n\ge 1,
\end{align}
and the sequence of polynomials $q_0=1$, and for $n\ge 1$
\begin{equation}\label{qnkr}
q_n(x)=k_{n}^{a,N}(x)+\beta_nk_{n-1}^{a,N}(x),
\end{equation}
where $k_{n}^{a,N}$ is the Krawtchouk polynomial (\ref{defkrp}) ($a\not =0,-1$).
Consider the second order difference operator $D_{a,N}$ (\ref{sodekr}) with respect to which
the Krawtchouk polynomials are eigenfunctions. Write finally $D$ for the difference operator of order $2k+2$ and genre $(-k-1,k+1)$
$$
D=P_1(D_{a,N})+\frac{1}{1+a}\nabla P_2(D_{a,N}).
$$
Then $D(q_n)=\lambda_nq_n$, $n\ge 0$.
\end{theorem}

For a convenient choice of the polynomial $P_1$, the polynomials $(q_n)_n$ in the previous Lemma turn out to be orthogonal with respect
to a moment functional. Indeed, for $k$ a positive integer, $a\not =0,-1$ and $N\not =0, -1 , \cdots , -k$ let $\tilde \rho_{k,a,N}$
be the measure defined by
\begin{equation}\label{tkrw}
\tilde \rho_{k,a,N}=(x+1)\cdots (x+k)\rho_{a,N+k+1}(x+k+1)
\end{equation}
where $\rho _{a,N}$ is the orthogonalizing moment functional  for the Krawtchouk polynomials $(k_{n}^{a,N})_n$.
For $N$ a positive integer, we have that
\begin{equation}\label{tkr2}
\tilde \rho_{k,a,N}=\frac{1}{(1+a)^{N+k}}
\left( (-1)^kk!\delta_{-k-1}+\sum _{x=0}^{N-1}\frac{\Gamma(N+k+1)a^{x+k+1}}{(x+k+1)\Gamma(N-x)x!}\delta _x \right) .
\end{equation}
To simplify the notation we remove the dependence of $a,N$ and $k$ and write $\tilde \rho =\tilde \rho_{k,a,N}$.

The proofs of the following theorems are similar to the ones in the previous Sections and are omitted.

\begin{lemma} For $k\ge 1$, $a\not =0,-1$, $N\in \RR$ and $n\ge 0$, we have
$$
\langle \tilde \rho_{k,a,N},k_{n}^{a,N}\rangle =\frac{(-1)^nk!k_{k}^{a,-N}(-n-1)}{(1+a)^{n}}.
$$
\end{lemma}

\begin{theorem}\label{lmkr}
For $k\ge 1$, let $a$ and $N$ be real numbers  satisfying $a\not =0,-1$
and $k_{k}^{a,-N}(-n)\not =0$, $n\ge 1$, where $k_{k}^{a,-N}$, $k\ge 1$, are Krawtchouk polynomials (see (\ref{defkrp})).
We define the sequences of numbers $(\lambda _n)_n$, $(\gamma _n)_{n\ge 1}$ and $(\beta _n)_{n\ge 1}$  by
\begin{align}\label{eigkrk}
\lambda _n&=-k_{k+1}^{a,-N+1}(-n),\\\label{defdelkr}
\gamma_n&=k_{k}^{a,-N}(-n),\\
\label{defbetkr}
\beta_n&=\frac{n}{1+a}\frac{\gamma_{n+1}}{\gamma_n},
\end{align}
and the sequence of polynomials $(q_n)_n$ by $q_0=1$, and
\begin{equation}\label{defqkr}
q_n(x)=k_{n}^{a,N}(x)+\beta _nk_{n-1}^{a,N}(x), \quad n\ge 1.
\end{equation}
If $N$ is not a positive integer, then the polynomials $(q_n)_n$ are orthogonal with respect
to the moment functional $\tilde \rho $ (\ref{tkrw}). If $N$ is a positive integer, then the polynomials $(q_n)_{0\le n\le N-1}$ are orthogonal
with respect
to the measure $\tilde \rho $ (\ref{tkr2}). Moreover, consider the difference operator of order $2k+2$ and genre $(-k-1,k+1)$ defined by
$$
D=P_1(D_{a,N})+\frac{1}{1+a}\nabla P_2(D_{a,N}),
$$
where $D_{a,N}$ is the second order difference operator for the Krawtchouk polynomials $(k_{n}^{a,N})_n$ (see \ref{sodekr}),
and  $P_1$ and $P_2$ are the polynomials of degrees $k+1$ and $k$, respectively, defined by
\begin{align}\label{defp1kr}
P_1(x)&=-k_{k+1}^{a,-N+1}\left( \frac{x}{1+a}\right) ,\\\label{defp2kr}
P_2(x)&=k_{k}^{a,-N}\left( \frac{x}{1+a}-1\right) .
\end{align}
Then $D(q_n)=\lambda_nq_n$.
\end{theorem}

Conjecture 2 in the Introduction for the Krawtchouk weight $\rho _{a,N}$
and the finite set $F=\{ 1, 2, \cdots, k\}$ can be deduced as a Corollary of the previous Theorem:

\begin{corollary}\label{lmkr3} For a positive integer $k$, write $F=\{ 1, 2, \cdots , k\}$ and $\pp(x)=\prod _{j\in F}(x-j).$
Assume that $a\not =0,-1$
and $k_{k}^{a,-N-k-1}(-n)\not =0$, $n\ge 1$. Then,
the orthogonal polynomials with respect to the moment functional $\pp\rho _{a,N}$ (\ref{tkrw}) are common eigenfunctions of a $(2k+2)$-order
difference operator of genre $(-k-1,k+1)$,
where $\rho_{a,N}$ is the moment functional for the Krawtchouk polynomials.
\end{corollary}

\section{Hahn case}
For $\alpha+c-N\not =-1, -2, \cdots$ we write $(h_{n}^{\alpha ,c,N})_n$ for the sequence of monic Hahn polynomials defined by
\begin{equation}\label{defhap}
h_{n}^{\alpha ,c,N}(x)=\sum_{j=0}^n\frac{(-n)_j(-x)_j(1-N+j)_{n-j}(c+j)_{n-j}}{(n+\alpha +c-N+j)_{n-j}j!}
\end{equation}
(we have taken a slightly different normalization from the one used in \cite{NSU}, pp. 30-53, from where
the next formulas can be easily derived; for more details see \cite{KLS}, pp, 204-211).
Hahn polynomials can be defined using the hypergeometric function ${}_3F_2$:
\begin{align}\nonumber
h_{n}^{\alpha ,c,N}(x)&=\frac{(c)_n(1-N)_n}{(n+\alpha+c-N)_n}\sum_{j=0}^n\frac{(-n)_j(-x)_j(n+\alpha +c-N)_{j}}{(1-N)_{j}(c)_{j}j!}
\\\label{defha3F2}
&=\frac{(c)_n(1-N)_n}{(n+\alpha+c-N)_n}\pFq{3}{2}{-n,-x,n+\alpha+c-N}{c,1-N}{1}.
\end{align}
Hahn polynomials are eigenfunctions of the following second order difference operator
\begin{align}\label{sodeha}
D_{\alpha,c,N}&=x(x-\alpha)\Sh_{-1}+(x+c)(x-N+1)\Sh_1\\ \nonumber &\hspace{1.5cm}+[-2x^2+(\alpha-c+N-1)x+\alpha+N(c-1)-1]\Sh_0,\\\nonumber
D_{\alpha,c,N}(h_{n}^{\alpha,c,N})&=(n+1)(n+\alpha+c-N-1)h_{n}^{\alpha,c,N},\quad n\ge 0.
\end{align}

They satisfy the following three term recurrence formula ($h_{-1}^{\alpha,c,N}=0$)
\begin{equation}\label{trha}
xh_n=h_{n+1}+b_nh_n+c_nh_{n-1}, \quad n\ge 0
\end{equation}
where
\begin{align*}
b_n&=\frac{c(N-1)(\alpha +c-N-1)+n(\alpha-c+N-1)(n+\alpha+c-N)}{(2n+\alpha +c-N-1)(2n+\alpha +c-N+1)},\\
c_n&=\frac{n(N-n)(n+\alpha +c-N-1)(n+\alpha-N)(n+c-1)(n+\alpha+c-1)}{(2n+\alpha +c-N-2)(2n+\alpha+c-N-1)^2(2n+\alpha +c-N)}
\end{align*}
(to simplify the notation we remove the parameters in some formulas).

Assume that $\alpha+c-N+1, \alpha-N+1, \alpha +c, c\not =0,-1, -2, \cdots $.
If, in addition, $N$ is not a positive integer, then the Hahn polynomials are always orthogonal with respect to
a moment functional $\rho_{\alpha,c,N}$, which we normalize by taking
$$
\langle \rho_{\alpha,c,N},1 \rangle=\frac{\Gamma(\alpha+1-N)\Gamma(c)\Gamma(\alpha+c)}{\Gamma(\alpha+c+1-N)}.
$$
When $N$ is a positive integer,  the first $N$ Hahn polynomials are orthogonal
with respect to the  Hahn measure
\begin{equation}\label{haw}
\rho_{\alpha,c,N}=\Gamma(N)\sum _{x=0}^{N-1} \frac{\Gamma(\alpha-x)\Gamma(x+c)}{\Gamma(N-x)x!}\delta _x
\end{equation}
(which it is positive when $\alpha>N-1$ and $c>0$).

We also need the so-called (monic) dual Hahn polynomials.
\begin{equation}\label{defdualhap}
h_{k}^{*,\alpha ,c,N}(x)=\sum_{j=0}^k\frac{(-k)_j(1-N+j)_{k-j}(c+j)_{k-j}}{j!}s_{j,N-\alpha-c}(x),
\end{equation}
where $s_{j,N-\alpha-c}$, $j\ge 0$, are the polynomials defined by (\ref{defsj}), that is, $s_{0,N-\alpha-c}=1$ and for $j\ge 1$
$$
s_{j,N-\alpha-c}(x)=(-1)^j\prod_{i=0}^{j-1}[x+i(N-\alpha-c-i)].
$$
Using Lemma \ref{nlp} we can also write the dual Hahn polynomials in terms of the hypergeometric function ${}_3F_2$
\begin{align}\nonumber
h_{k}^{*,\alpha,c,N}(x(x+\alpha+c-N))&=(c)_k(1-N)_k\sum_{j=0}^k\frac{(-k)_j(-x)_j(x+\alpha +c-N)_{j}}{(1-N)_{j}(c)_{j}j!}
\\\label{defdha3F2}
&=(c)_k(1-N)_k\pFq{3}{2}{-k,-x,x+\alpha+c-N}{c,1-N}{1}.
\end{align}
(\ref{defha3F2}) and (\ref{defdha3F2}) give the duality of the Hahn and dual Hahn polynomials with respect to the sequences
$n$ and $n(n+\alpha+c-N)$, that is, for $k,n\ge 0$
$$
(c)_n(1-N)_nh_k^{*,\alpha,c,N}(n(n+\alpha+c-N))=(c)_k(1-N)_k(n+a+c-N)_nh_n^{\alpha,c,N}(k).
$$

The computations in the Hahn case are technically more involved because
the eigenvalue sequence $(\theta _n)_n$ (\ref{sodeha}) for the Hahn polynomials (with respect to its  second order difference operator)
is quadratic in $n$.
We first need some notation.

We write $(\sigma_n)_n$ and $(\theta_n)_n$ for the sequences
\begin{align}\label{osigj}
\theta_n&=(n+1)(n+\alpha+c-N-1), \\\label{osigj2}
\sigma_n&=2n+\alpha+c-N-2,
\end{align}
($(\sigma_n)_n$ is
one of sequences which define the first and four $\D$-operators in Lemma \ref{lTha} below).
For $j\ge 0$, we also consider the sequence
$$
u_j(n)=(n)_j(-n-\alpha -c+N+2)_j,
$$
and the polynomials $r_0=1$ and for $j\ge 1$
\begin{equation}\label{defrj}
r_j(x)=(-1)^j\prod _{i=0}^{j-1}[x+i(\alpha +c-N-2-i)].
\end{equation}
Since the polynomial $r_j$ has degree just $j$, all polynomial
$P\in  \PP$ can be written as $P(x)=\sum_{j=0}^{\deg(P)}w_jr_j(x)$ for certain numbers $w_j$, $j=0,\cdots , \deg(P)$.

We will need the following technical Lemma.

\begin{lemma}\label{ulll}
Let $\alpha, c, N \in \RR$ be real numbers. Then for $j,n\ge 0$
\begin{align} \label{l721}
u_j(n)&=r_{j}(\theta_{n-1}), \\\label{l722}
\sigma _nu_j(n)+\sigma _{n+1}u_j(n+1)&=-2\frac{u_{j+1}(n+1)-u_{j+1}(n)}{j+1}\\\nonumber &\hspace{-1cm}+(-\alpha-c+N+2(j+1))(u_{j}(n+1)-u_{j}(n)).
\end{align}
\end{lemma}

\begin{proof}

(\ref{l721}) is just Lemma \ref{nlp} for $x=-n$ and $u=\alpha +c-N-2$.

(\ref{l722}) of the Lemma can be checked by a direct computation.

\end{proof}

\bigskip

Proceeding as in the previous Sections, we have identified four $\D$-operators of type 2 for Hahn polynomials (the proof is omitted).

\begin{lemma}\label{lTha}
For $\alpha , c, N$ satisfying $\alpha+c-N\not =-1, -2, \cdots$, consider the Hahn polynomials $(h_{n}^{\alpha ,c,N})_n$ (\ref{defhap}).
Then, the operators $\D _i$, $i=1, 2, 3, 4$ defined by (\ref{defTo2}) from the sequences ($n\ge 0$)
\begin{align*}
\varepsilon _{n,1}&=\frac{n(N-n)(n+\alpha-N)}{(2n+\alpha+c-N-1)(2n+\alpha+c-N-2)},\quad &\sigma_{n,1}&=2n+\alpha+c-N-2,\\
\varepsilon _{n,2}&=\frac{n(n+\alpha-N)(n+\alpha+c-1)}{(2n+\alpha+c-N-1)(2n+\alpha+c-N-2)},&\sigma_{n,2}&=-(2n+\alpha+c-N-2),\\
\varepsilon _{n,3}&=\frac{-n(N-n)(n+c-1)}{(2n+\alpha+c-N-1)(2n+\alpha+c-N-2)}, &\sigma_{n,3}&=-(2n+\alpha+c-N-2),\\
\varepsilon _{n,4}&=\frac{-n(n+c-1)(n+\alpha+c-1)}{(2n+\alpha+c-N-1)(2n+\alpha+c-N-2)}, &\sigma_{n,4}&=2n+\alpha+c-N-2,
\end{align*}
are $\D$-operators for $(h_{n}^{\alpha ,c,N})_n$ and the algebra $\A$  (\ref{algdiff}) of difference
operators  with polynomial coefficients.
More precisely
\begin{align*}
\displaystyle \D_1&=(-x+N-1)\Delta-\frac{\alpha+c-N}{2}I, \\
\displaystyle \D_2&=(x-\alpha)\nabla+\frac{\alpha+c-N}{2}I,\\
\displaystyle \D_3&=x\nabla+\frac{\alpha+c-N}{2}I,\\
\displaystyle \D_4&=-(x+c)\Delta-\frac{\alpha+c-N}{2}I.
\end{align*}

\end{lemma}

Each one of the $\D$-operators displayed in the previous Lemma together with Lemma \ref{fl1v} gives rise to the corresponding
class of polynomials satisfying higher order difference equations.
We only proceed with the first and second $\D$-operators because due to the symmetries of Hahn polynomials,
the examples generated from the first and third operators are
essentially the same, as well as the examples generated from the second and the fourth operators.

\begin{theorem}\label{th5.1ha} Let $\alpha , c, N$ be real numbers satisfying $\alpha+c-N\not =-1,-2,\cdots $.
Let $P_2$ be an arbitrary polynomial of degree
 $k\ge 1$ which we write in the form
$$
P_2(x)=\sum_{j=0}^{k}w_jr_j(x),
$$
for certain numbers $w_j$, $j=0,\cdots , k$. Consider also the polynomial $P_1$ of degree $k+1$ defined by
$$
P_1(x)=(\alpha+c-N)P_2(x)+2(x-\alpha-c+N+1)\sum_{j=0}^k\frac{w_j}{j+1}r_j(x).
$$
Consider next de sequences of numbers $(\gamma_n)_n$, $(\beta_{n,i})_n$, $i=1,2$, and $(\lambda_n)_n$ defined by
\begin{align}\nonumber
\gamma_n&=\sum _{j=0}^k w_ju_j(n),\quad n\ge 0,\\\label{obnj}
\beta_{n,1}&=\frac{n(N-n)(n+\alpha-N)}{(2n+\alpha+c-N-1)(2n+\alpha+c-N-2)}\frac{\gamma_{n+1}}{\gamma_n},\quad n\ge 1,\\\label{obnj2}
\beta_{n,2}&=\frac{n(n+\alpha-N)(n+\alpha+c-1)}{(2n+\alpha+c-N-1)(2n+\alpha+c-N-2)}\frac{\gamma_{n+1}}{\gamma_n},\quad n\ge 1,\\\label{lnj}
 \lambda_n&=\frac{\sigma _n\gamma_n+P_1(\theta_{n-1})}{2}, \quad n\ge 1,\quad \lambda_0=\frac{P_1(\theta_0)-\sigma_1P_2(\theta_0)}{2},
\end{align}
where $(\theta_n)_n$ and $(\sigma_n)_n$ are defined by (\ref{osigj}) and (\ref{osigj2}), and
we implicitly assume that $\gamma_n\not =0$, $n\ge 1$. Define finally the sequences of polynomials $q_{0,i}=1$, and for $n\ge 1$
\begin{equation}\label{qnnoj}
q_{n,i}=h_{n}^{\alpha,c,N}+\beta_{n,i}h_{n-1}^{\alpha,c,N},\quad i=1,2,
\end{equation}
where $h_{n}^{\alpha,c,N}$ is the Hahn polynomial (\ref{defhap}).
Consider the second order difference operator $D_{\alpha,c,N}$ (\ref{sodeha}) with respect to which
the polynomials $(h_{n}^{\alpha,c,N})_n$ are eigenfunctions. Write finally $D_i$, $i=1,2$, for the difference operators of order $2k+2$
and genre $(-k-1,k+1)$
\begin{align*}
D_1&=\frac{1}{2}P_1(D_{\alpha,c,N})+\left(-\frac{\alpha+c-N}{2}+(-x+N-1)\Delta \right) P_2(D_{\alpha,c,N}),\\
D_2&=\frac{1}{2}P_1(D_{\alpha,c,N})-\left(\frac{\alpha+c-N}{2}+(x-\alpha )\nabla \right) P_2(D_{\alpha,c,N}).
\end{align*}
Then $D_i(q_{n,i})=\lambda_nq_{n,i}$, $n\ge 0$, $i=1,2$.
\end{theorem}

\begin{proof}

The theorem is  a straightforward consequence of the Lemma \ref{fl2v}.
We have just to identify who the main characters are in this example.

Write
\begin{align*}
\varepsilon_{n,1}&=\frac{n(N-n)(n+\alpha-N)}{(2n+\alpha+c-N-1)(2n+\alpha+c-N-2)},\quad \sigma_{n,1}=\sigma_n,\\
\varepsilon_{n,2}&=\frac{n(n+\alpha-N)(n+\alpha+c-1)}{(2n+\alpha+c-N-1)(2n+\alpha+c-N-2)},\quad \sigma_{n,2}=-\sigma_n.
\end{align*}
Lemma \ref{lTha} gives that the sequences $(\varepsilon_{n,i})_n$ and $(\sigma_{n,i})_n$, $i=1,2$,
define two $\D$-operators of type 2 (\ref{defTo2}) for the polynomials $(h_{n}^{\alpha,c,N})_n$
and that
\begin{align*}
\D_1&=-\frac{\alpha+c-N}{2}I+(-x+N-1)\Delta ,\\
\D_2&=\frac{\alpha+c-N}{2}I+(x-\alpha)\nabla.
\end{align*}

Consider now the polynomial
\begin{equation}\label{alP1}
Q(x)=\sum _{j=0}^k w_j\left( \frac{-2}{j+1}r_{j+1}(x)+(-\alpha-c+N+2(j+1))r_j(x)\right).
\end{equation}
A straightforward computation using (\ref{defrj}) gives that $Q=P_1$.

According to the  Lemma \ref{fl2v} we have just to check that
\begin{enumerate}
\item $\gamma_{n+1}=P_2(\theta_n)$, $n\ge 0$.
\item $\lambda_n-\lambda_{n-1}=\sigma_n\gamma_n$, $n\ge 1$.
\item $\lambda_n+\lambda_{n-1}=P_1(\theta_{n-1})$, $n\ge 1$, where $\theta_n$ are the eigenvalues of $(h_{n}^{\alpha,c,N})_n$
with respect to $D_{\alpha,c,N}$.
\end{enumerate}

Formula (1) is a direct consequence of (\ref{l721}) in Lemma \ref{ulll}.

Using the definition of $\lambda _n$, one can easily see  that for $n\ge 2$ conditions 2 and 3 above are equivalent to
$$
\frac{P_1(\theta_n)-P_1(\theta_{n-1})}{2}=\frac{\sigma_n\gamma_n+\sigma_{n+1}\gamma_{n+1}}{2}.
$$
But this is a direct consequence of (\ref{l722}) in Lemma \ref{ulll} and the formula (\ref{alP1}) for $P_1$.

Conditions 2 and 3 above for $n=1$ can be checked by a straightforward computation.
\end{proof}

\subsection{Hahn I}
For a convenient choice of the polynomial $P_2$, the polynomials $(q_{n,1})_n$ in the Theorem \ref{th5.1ha} turn out to be orthogonal with respect
to a moment functional. To identify that polynomial $P_2$, we need  the family of dual Hahn polynomials
$h_{1,k}^*=h_k^{*,N+c-1,2-c,\alpha +c-1}$. Using (\ref{defdualhap}) and (\ref{defdha3F2}) we get
\begin{align}\label{dhp}
h^*_{1,k}(x)&=\sum _{j=0}^k\frac{(-k)_j(2-\alpha-c+j)_{k-j}(2-c+j)_{k-j}}{j!}r_j(x),\\\nonumber
h^*_{1,k}(x(x+2+N&-\alpha-c))\\\nonumber
&=(2-\alpha-c)_k(2-c)_k\pFq{3}{2}{-k,-x,x+2+N-\alpha-c}{2-\alpha-c,2-c}{1} .
\end{align}
This coincides with the hypergeometric representation for the (monic) dual Hahn polynomials given in \cite{KLS}, p. 208 (with
the notation of \cite{KLS},
$$
h^*_{1,k}(\lambda (x))=(2-\alpha-c)_k(2-c)_kR_k(\lambda (x),1-\alpha -c, N, c-2)
$$
where $\lambda(x)=x(x+2+N-\alpha-c)$).
They satisfy the following second order difference equation in the lattice $x(x+2+N-\alpha-c)$ (see \cite{KLS} pp. 209)
\begin{equation}\label{sodedh}
r(x)\tilde \Sh _{-1}(h^*_{1,k})-(r(x)+s(x))\tilde \Sh _{0}(h^*_{1,k})+s(x)\tilde \Sh _{1}(h^*_{1,k})=ku(x)\tilde \Sh _{0}(h^*_{1,k}),
\end{equation}
where $\tilde \Sh _{l}(p)=p((x+l)(x+l+2+N-\alpha-c))$, $l\in \ZZ$, and
\begin{align}\nonumber
r(x)&= -x(x+N)(x-\alpha+N)(2x-\alpha-c+N+3),\\\label{rsu}
s(x)&=(x-\alpha -c+N+2)(x-\alpha -c +2)(x-c+2)(2x-\alpha-c+N+1) ,\\ \nonumber
u(x)&=(2x-\alpha-c+N+1)(2x-\alpha-c+N+2)(2x-\alpha-c+N+3).
\end{align}
They also satisfy the following first order difference equation (\cite{KLS} pp. 210): if we write $\tilde \Delta=\tilde \Sh _{1}-\tilde \Sh _{0}$, then
\begin{equation}\label{fodedh}
\tilde \Delta (h_{1,k}^{*,c})=k(2x-\alpha-c+N+3)h_{1,k-1}^{*,c-1}(x(x+3+N-\alpha-c)).
\end{equation}
\bigskip
For $k$ a positive integer, and $\alpha, c, N\in \RR$ satisfying
\begin{equation}\label{conacN}
\alpha +c-N+1, \alpha-N+1,\alpha+c-k-1, c-k-1\not\not =0,-1,-2,\cdots ,
\end{equation}
let $\tilde \rho_{k,\alpha ,c,N}$ be the moment functional defined by
\begin{equation}\label{thaw}
\tilde \rho_{k,\alpha ,c,N}=(x+c-1)\cdots (x+c-k)\rho_{\alpha ,c-k-1,N}
\end{equation}
where $\rho_{\alpha ,c,N}$ is the orthogonalizing moment functional for the Hahn polynomials $(h_{n}^{\alpha ,c,N})_n$.
When $N$ is a positive integer we have that
\begin{equation}\label{thaw2}
\tilde \rho_{k,\alpha ,c,N}=\Gamma(N)\sum _{x=0}^{N-1} \frac{\Gamma(\alpha-x)\Gamma(x+c)}{(x+c-k-1)\Gamma(N-x)x!}\delta _x,
\end{equation}
(only the first $N$ Hahn polynomials are then orthogonal).
To simplify the notation we remove the dependence of $\alpha , c, N$ and $k$ and write $\tilde \rho =\tilde \rho_{k,\alpha ,c,N}$.

In order to find the orthogonal polynomials with respect to $\tilde \rho _{k,\alpha ,c,N}$ we need the following Lemma.

\begin{lemma} For $k$ a positive integer, let $\alpha, c, N$ be real numbers satisfying
(\ref{conacN}). Then, we have
$$
\langle \tilde\rho_{k,\alpha ,c,N},h_{n}^{\alpha ,c,N}\rangle =
\frac{(-1)^nn!(N-n)_n\Gamma(c-k-1)\Gamma(\alpha+n-N+1)\Gamma(\alpha+c-k-1)h^*_{1,k}(\theta_n)}
{\Gamma(\alpha+c-N+2n)},
$$
where $(\theta_n)$ is the sequence defined by (\ref{osigj}).
\end{lemma}

\begin{proof}
We proceed in an analogous way to the Lemmas \ref{chxx} or \ref{lme1x}.
For $n\ge 0$ write
\begin{align*}
\xi_n&=\langle \tilde\rho_{k,\alpha ,c,N},h_{n}^{\alpha ,c,N}\rangle , \\
\zeta_n&=\frac{(-1)^nn!(N-n)_n\Gamma(c-k-1)\Gamma(\alpha+n-N+1)\Gamma(\alpha+c-k-1)h^*_{1,k}(\theta_n)}
{\Gamma(\alpha+c-N+2n)}.
\end{align*}
The three term recurrence relation (\ref{trha}) for $(h_{n}^{\alpha,c,N})_n$ and the normalization for the Hahn weight $\rho_{\alpha,c,N}$ give that
\begin{align}\label{ultdva}
\xi_{1}+(b_0+c-k-1)\xi_{0}-\frac{\Gamma(c)\Gamma(\alpha+1-N)\Gamma(\alpha+c)}{\Gamma(\alpha+c+1-N)}&=0, \\
\label{ultdv}
\xi_{n+1}+(b_n+c-k-1)\xi_{n}+c_n\xi_{n-1}=0, \quad n\ge 1,
\end{align}
where $b_n$ and $c_n$ are the recurrence coefficients of the Hahn polynomials $(h_{n}^{\alpha,c,N})_n$.

We now prove that also
\begin{align}\label{mecme3a}
\zeta_{1}+(b_0+c-k-1)\zeta_{0}-\frac{\Gamma(c)\Gamma(\alpha+1-N)\Gamma(\alpha+c)}{\Gamma(\alpha+c+1-N)}&=0, \\
\label{mecme3}
\zeta_{n+1}+(b_n+c-k-1)\zeta_{n}+c_n\zeta_{n-1}=0, \quad n\ge 1.
\end{align}
After straightforward computations taking into account the definition of $\zeta _n$, (\ref{mecme3}) is equivalent to
\begin{align*}
r(-n-1)h^*_{1,k}(\theta _{n+1})-&(r(-n-1)+s(-n-1)+ku(-n-1))h^*_{1,k}(\theta _{n})\\&+s(-n-1)h^*_{1,k}(\theta _{n-1})=0
\end{align*}
where the polynomials $r, s$ and $u$ are given by (\ref{rsu}). But this follows just by writing  $x=-n-1$
in the second order difference equation (\ref{sodedh}) for the dual Hahn polynomials $(h^*_{1,k})_k$. (\ref{mecme3a}) follows in
a similar way.

Since the sequences $(\xi_n)_n$ and $(\zeta_n)_n$ satisfy the same recurrence relation, it is enough to prove
that $\xi_0=\zeta_0$. If\ we write
$$
\eta_{k,c}=\xi_0=\langle \tilde\rho_{k,\alpha ,c,N},1\rangle =\langle \rho_{\alpha ,c,N},(x+c-1)\cdots (x+c-k)\rangle ,
$$
and proceed as in Lemma \ref{lme1x}, we find that
$$
\eta_{k,c}-k\eta_{k-1,c-1}=\frac{\Gamma(c-1)\Gamma(\alpha-N+1)\Gamma(\alpha+c-1)}
{\Gamma(\alpha+c-N)}.
$$
On the other hand, if we write
$$
\tau _{k,c}=\zeta_0=\frac{\Gamma(c-k-1)\Gamma(\alpha-N+1)\Gamma(\alpha+c-k-1)h_{1,k}^{*,c}(\alpha+c-N-1)}
{\Gamma(\alpha+c-N)},
$$
a simple computation using (\ref{fodedh}) gives
$$
\tau_{k,c}-k\tau_{k-1,c-1}=\frac{\Gamma(c-1)\Gamma(\alpha-N+1)\Gamma(\alpha+c-1)}
{\Gamma(\alpha+c-N)}.
$$
From where we easily get $\xi_0=\zeta_0$.

\end{proof}

We are now ready to prove that the orthogonal polynomials with respect to $\tilde \rho _{k,a,c,N}$ (\ref{thaw}) are a particular case
of Theorem \ref{th5.1ha}.

\begin{theorem} \label{lm61ha}
For $k\ge 1$, let $\alpha, c, N$ be real numbers satisfying (\ref{conacN})
and $h^*_{1,k}(n(n-2-N+a+c))\not =0$, $n\ge 0$, where $h_{1,k}^*$, $k\ge 1$, are the dual Hahn polynomials (\ref{dhp}).
We then define the polynomials $P_2$ and $P_1$, of degrees $k$ and $k+1$, respectively, by
\begin{align}\label{defp1ha}
P_2(x)&=h^*_{1,k}(x) ,\\\label{defp2ha}
P_1(x)&=(\alpha +c-N)h^*_{1,k}(x)+2(x-\alpha-c+N+1)\\ \nonumber
&\hspace
{2cm} \times\sum _{j=0}^k\frac{(-k)_j(2-\alpha-c+j)_{k-j}(2-c+j)_{k-j}}{(j+1)!}r_j(x).
\end{align}
We also define the sequences of numbers $(\lambda _n)_n$, $(\gamma _n)_{n\ge 1}$ and $(\beta _n)_{n\ge 1}$  by
\begin{align}\label{eighak}
 \lambda_n&=\frac{\sigma _n\gamma_n+P_1(\theta_{n-1})}{2}, \quad n\ge 1,\quad \lambda_0=
 \frac{P_1(\theta_0)-\sigma_1P_2(\theta_0)}{2},\\\label{defdelha}
\gamma_n&=h^*_{1,k}(n(n-2-N+a+c)),\\
\label{defbetha}
\beta_n&=\frac{n(N-n)(n+\alpha-N)}{(2n+\alpha+c-N-1)(2n+\alpha+c-N-2)}\frac{\gamma_{n+1}}{\gamma_n},
\end{align}
and the sequence of polynomials $(q_n)_n$ by $q_0=1$, and
\begin{equation}\label{defqha}
q_n(x)=h_{n}^{a,c,N}(x)+\beta _nh_{n-1}^{a,c,N}(x), \quad n\ge 1,
\end{equation}
where $(\theta_n)_n$ and $(\sigma_n)_n$ are defined by (\ref{osigj}) and (\ref{osigj2}).
 If $N$ is not a positive integer, then the polynomials $(q_n)_n$ are orthogonal with respect
to the moment functional $\tilde \rho $ (\ref{thaw}). If $N$ is a positive integer, then the polynomials
$(q_n)_{0\le n\le N}$ are orthogonal with respect to the measure $\tilde \rho$ (\ref{thaw2}). Moreover,
if we consider the difference operator of order $2k+2$ and genre $(-k-1,k+1)$ defined by
$$
D=\frac{1}{2}P_1(D_{\alpha,c,N})+\left(-\frac{\alpha+c-N}{2}+(-x+N-1)\Delta \right) P_2(D_{\alpha,c,N}),
$$
where $D_{\alpha ,c,N}$ is the second order difference operator for the Hahn polynomials (\ref{sodeha}), then
$D(q_n)=\lambda_nq_n$.
\end{theorem}

\begin{proof}
The first part of the Theorem can be proved in an analogous way to Theorem \ref{lm51}.

The second part is the particular case of  Theorem \ref{lm61ha} for $i=1$  and
$$
w_j=\frac{(-k)_j(2-\alpha-c+j)_{k-j}(2-c+j)_{k-j}}{j!}.
$$
\end{proof}

Conjecture 1 in the Introduction for the Hahn weight $\rho _{\alpha ,c,N}$
and the finite set $F=\{ 1, 2, \cdots, k\}$ can be deduced as a Corollary of the previous Theorem:

\begin{corollary}\label{cha555} For a positive integer $k$, write $F=\{ 1, 2, \cdots , k\}$ and $\ps(x)=\prod _{j\in F}(x+c+j).$ Let $\alpha , c,N$ be
real numbers satisfying that $\alpha+c-N+k+2,\alpha -N+1, \alpha +c, c\not =0, -1, -2, \cdots$
and that
$$
h_{1,k}^{*,\alpha,c+k+1,N}(n(n+\alpha+c+k-N-1))\not =0.
$$
Then, the orthogonal polynomials with respect to the measure $\ps\rho _{a,c,N}$
are common eigenfunctions of a $(2k+2)$-order difference operator of genre $(-k-1,k+1)$.
\end{corollary}

\subsection{Hahn II}
There is other choice of the polynomial $P_2$ for which the polynomials $(q_{n,2})_n$ in the Theorem \ref{th5.1ha} are also orthogonal with respect
to a moment functional. Indeed, write $(h^*_{2,k})_k$  for the family of dual Hahn polynomials
$h^*_{2,k}=h_{k}^{*,-\alpha,2-c,-N}$. Using (\ref{defdualhap}) and (\ref{defdha3F2}) we get
\begin{align}\label{dhp2}
h^*_{2,k}(x)&=\sum _{j=0}^k\frac{(-k)_j(2-c+j)_{k-j}(N+1+j)_{k-j}}{j!}r_j(x),\\
h^*_{2,k}(x(x+2&+N-\alpha-c))\\ \nonumber
&=(2-c)_k(N+1)_k\pFq{3}{2}{-k,-x,x+2+N-\alpha-c}{2-c,N+1}{1}.
\end{align}
This coincides with the hypergeometric representation for the dual Hahn polynomials given in \cite{KLS}, p. 208 (with
the notation of \cite{KLS},
$$
h^*_{2,k}(\lambda (x))=(2-c)_k(N+1)_kR_k(\lambda (x),1-c, N-\alpha, -N-1)
$$
where $\lambda(x)=x(x+2+N-\alpha-c)$).

\bigskip
For $k$ a positive integer, and $\alpha, c, N\in \RR$ satisfying
\begin{equation}\label{conacN2}
\alpha +c-N+1, \alpha -N+1,\alpha+c,c-k-1\not =0,-1,-2,\cdots ,
\end{equation}
let $\tilde \rho_{k,\alpha ,c,N}$ be the measure defined by
\begin{equation}\label{thawb}
\tilde \rho_{k,\alpha ,c,N}=(x+1)\cdots (x+k)\rho_{\alpha +k+1,c-k-1,N+k+1}(x+k+1).
\end{equation}
When $N$ is a positive integer we have
\begin{align}\label{thawb2}
\tilde \rho_{k,\alpha ,c,N}=(-1)^kk!&\Gamma(\alpha +k+1)\Gamma(c-k-1)\delta_{-k-1}\\ \nonumber &+
\Gamma(N+k+1)\sum _{x=0}^{N-1} \frac{\Gamma(\alpha-x)\Gamma(x+c)}{(x+k+1)\Gamma(N-x)x!}\delta _x.
\end{align}
To simplify the notation we remove the dependence of $a, c, N$ and $k$ and write $\tilde \rho =\tilde \rho_{k,a,c,N}$.

The proofs of the following Theorems are similar to that of the previous Sections and are omitted.

\begin{lemma} For $k$ a positive integer, let $\alpha, c, N$ be real numbers satisfying
(\ref{conacN2}). Then, we have
$$
\langle \tilde \rho_{k,\alpha ,c,N},h_{n}^{\alpha ,c,N}\rangle =
\frac{(-1)^{n+k}n!\Gamma(c-k-1)\Gamma(\alpha+c+n)\Gamma(\alpha+1+n-N)h^*_{2,k}(\theta_n)}{\Gamma(\alpha+c-N+2n)},
$$
where $(\theta_n)$ is the sequence defined by (\ref{osigj}).
\end{lemma}

\begin{theorem} \label{lm61ha2}
For $k\ge 1$, let $\alpha, c, N$ be real numbers satisfying (\ref{conacN2})
and $h^*_{2,k}(n(n-2-N+a+c))\not =0$, $n\ge 0$, where $h_{2,k}^*$, $k\ge 1$, are the dual Hahn polynomials (\ref{dhp2}).
We then define the polynomials $P_2$ and $P_1$, of degrees $k$ and $k+1$, respectively, by
\begin{align}\label{defp1ha2}
P_2(x)&=h^*_{2,k}(x) ,\\\label{defp2ha2}
P_1(x)&=(\alpha +c-N)h^*_{2,k}(x)+2(x-\alpha-c+N+1)\\ \nonumber
&\hspace
{2cm} \times\sum _{j=0}^k\frac{(-k)_j(2-c+j)_{k-j}(N+1+j)_{k-j}}{(j+1)!}r_j(x).
\end{align}
We also define the sequences of numbers $(\lambda _n)_n$, $(\gamma _n)_{n\ge 1}$ and $(\beta _n)_{n\ge 1}$  by
\begin{align}\label{eigmek2}
 \lambda_n&=\frac{\sigma _n\gamma_n+P_1(\theta_{n-1})}{2}, \quad n\ge 1,\quad \lambda_0
 =\frac{P_1(\theta_0)-\sigma_1P_2(\theta_0)}{2},\\\label{defdelha2}
\gamma_n&=h^*_{2,k}(n(n-2-N+a+c)),\\
\label{defbetha2}
\beta_n&=\frac{n(n+\alpha -N)(n+\alpha+c-1)}{(2n+\alpha+c-N-1)(2n+\alpha+c-N-2)}\frac{\gamma_{n+1}}{\gamma_n},
\end{align}
and the sequence of polynomials $(q_n)_n$ by $q_0=1$, and
\begin{equation}\label{defqha2}
q_n(x)=h_{n}^{a,c,N}(x)+\beta _nh_{n-1}^{a,c,N}(x), \quad n\ge 1,
\end{equation}
where $(\theta_n)_n$ and $(\sigma_n)_n$ are defined by (\ref{osigj}) and (\ref{osigj2}).
If $N$ is not a positive integer, then the polynomials $(q_n)_n$ are orthogonal with respect
to the moment functional $\tilde \rho $ (\ref{thawb}). If $N$ is a positive integer, then the polynomials
$(q_n)_{0\le n\le N}$ are orthogonal with respect to the measure $\tilde \rho$ (\ref{thawb2}). Moreover, if we
consider the difference operator of order $2k+2$ and genre $(-k-1,k+1)$ defined by
$$
D=\frac{1}{2}P_1(D_{\alpha,c,N})-\left(\frac{\alpha+c-N}{2}+(x-\alpha)\nabla \right) P_2(D_{\alpha,c,N}),
$$
where $D_{\alpha ,c,N}$ is the second order difference operator for the Hahn polynomials (\ref{sodeha}), then
$D(q_n)=\lambda_nq_n$.
\end{theorem}

Conjecture 2 in the Introduction for the Hahn weight $\rho _{\alpha ,c,N}$
and the finite set $F=\{ 1, 2, \cdots, k\}$ can be deduced as a Corollary of the previous Theorem:

\begin{corollary}\label{hahaha} For a positive integer $k$, write $F=\{ 1, 2, \cdots , k\}$ and $\pp(x)=\prod _{j\in F}(x-j).$
Let $\alpha , c, N$ be real numbers satisfying $\alpha+c-N+k+2, \alpha -N+1, \alpha+c, c\not =0,-1,\cdots$,
and that
$$
h_{2,k}^{*,\alpha-k-1,c+k+1,N-k-1}(n(n+\alpha+c-N+k-1))\not =0.
$$
Then, the orthogonal polynomials with respect to the moment functional $\pp\rho _{a,c,N}$
are common eigenfunctions of a $(2k+2)$-order difference operator of genre $(-k-1,k+1)$.
\end{corollary}

\section{Appendix}

In this Appendix, we apply our method to the Laguerre and Jacobi polynomials to generate families of orthogonal polynomials
which are eigenfunctions of higher order differential operators.

\subsection{Laguerre polynomials}
For $\alpha\in \RR $, we use the standard definition of the Laguerre polynomials $(L_{n}^{\alpha})_n$
\begin{equation}\label{deflap}
L_{n}^{\alpha}(x)=\sum _{j=0}^n \frac{(-x)^{j}}{j!}\binom{n+\alpha}{n-j}
\end{equation}
(that and the next formulas can be found in \cite{EMOT}, pp. 188-192).

When $\alpha \not =-1,-2,\cdots $, they are  orthogonal with respect to a measure $\mu_\alpha=\mu_\alpha(x)dx$,
which it is positive only when $\alpha >-1$,
and then
\begin{equation}\label{Laguerrew}
\mu_\alpha (x) =x^\alpha e^{-x}, \quad 0<x.
\end{equation}
Laguerre polynomials are eigenfunctions of the following second order differential operator
\begin{equation}\label{sodolag}
D_{\alpha} =x\left(\frac{d}{dx}\right) ^2+(\alpha +1-x)\frac{d}{dx} ,\qquad D_{\alpha} (L_{n}^{\alpha})=-nL_{n}^{\alpha},\quad n\ge 0.
\end{equation}
We first identify a $\D$-operator for the Laguerre polynomials.

\begin{lemma}\label{lTlag}
For $\alpha \in \RR $, let $p_{n}=L_{n}^{\alpha }$, $n\ge 0$, the Laguerre polynomials given by (\ref{deflap}).
Then the operator $\D$ defined by (\ref{defTo}) from the sequence $\varepsilon _n=-1$, $n\ge 0$,
is a $\D$-operator for the Laguerre polynomials and the algebra $\A$  (\ref{algdiff1}).
More precisely  $\D=d/dx$.
\end{lemma}

\begin{proof}
Since $p_{n}=L_{n}^{\alpha }$, using the definition above of the operator $\D$ and the well-known formulas
$(L_{n}^{\alpha})'=-L_{n-1}^{\alpha +1}$, $L_{n}^{\alpha+1}=\sum_{j=0}^nL_{n-j}^{\alpha}$, we have
\begin{align*}
\D(L_{n}^{\alpha})&=\sum_{j=1}^\infty (-1)^{j+1}\zeta ^j(L_{n}^{\alpha})=\sum_{j=1}^n (-1)^{j+1}(-1)^jL_{n-j}^{\alpha}=\\
&=-\sum_{j=1}^n L_{n-j}^{\alpha}=-\sum_{j=0}^{n-1} L_{n-1-j}^{\alpha}=-L_{n-1}^{\alpha+1}=\frac{d(L_{n}^{\alpha})}{dx}.
\end{align*}
This gives $\D=d/dx$.

\end{proof}

The previous Lemma together with Lemma \ref{fl1v} gives rise to the corresponding families of polynomials satisfying higher order differential equations.

\begin{theorem}\label{th5.1lag} Let $P_1$ be a polynomial of degree $k+1$, $k\ge 1$,
and write $P_2(x)=P_1(x-1)-P_1(x)$ (so that $P_2$ is a polynomial of degree $k$). We assume that $P_2(-n)\not =0$, $n\ge 0$, and
define the sequences of numbers
\begin{align}\nonumber
\gamma_{n+1}&=P_2(-n),\quad n\ge 0,\\\nonumber
\lambda_n&=P_1(-n),\quad n\ge 0,\\\label{becc}
\beta_n&=-\frac{\gamma_{n+1}}{\gamma_n}\quad n\ge 1,
\end{align}
and the sequence of polynomials $q_0=1$, and for $n\ge 1$
\begin{equation}\label{qnnolag}
q_n=L_{n}^{\alpha}+\beta_nL_{n-1}^{\alpha},
\end{equation}
where $L_{n}^{\alpha}$ is the $n$-th Laguerre polynomial.
Write finally $D$ for the differential operator of order $2k+2$
$$
D=P_1(D_\alpha)+\frac{d}{dx} P_2(D_\alpha),
$$
where $D_\alpha$ is the second order differential operator  with respect to which
the Laguerre polynomials $(L_{n}^{\alpha})_n$ are eigenfunctions.

Then $D(q_n)=\lambda_nq_n$, $n\ge 0$.
\end{theorem}

The case $k=1$ in the previous Theorem is included in \cite{GrH3}, and for general $k$ is implicit in \cite{Plamen2}.

We next consider some important particular cases of the previous theorem.

For $\alpha \not =0,-1,-2, \cdots $ and $M\not =-\Gamma(n)/\Gamma(n+\alpha)$, consider
the measure $\tilde \rho_{\alpha,M}$ defined by
\begin{equation}\label{darlag}
\tilde \rho_{\alpha,M}=\alpha \Gamma^2(\alpha )M\delta _0+\mu _{\alpha-1}(x) dx, \quad 0<x,
\end{equation}
where $\mu_\alpha(x)$ is the orthogonalizing function for the Laguerre polynomials $(L_{n}^{\alpha})_n$.
When $\alpha>0$ and $M>0$, the measure $\tilde \rho _{\alpha, M}$ is the
Krall-Laguerre-Koornwinder measure
$$
\alpha \Gamma^2(\alpha )M\delta _0+x^{\alpha -1}e^{-x}, \quad 0<x.
$$
To simplify the notation we remove the dependence of $\alpha$ and $M$ and write $\tilde \rho =\tilde \rho_{\alpha,M}$.

Using Lemma \ref{addel}, we
 find explicitly a sequence of orthogonal polynomials with respect to $\tilde \rho $
in terms of $p_{n}=L_{n}^{\alpha }$, $n\ge 0$.

\begin{lemma}\label{l4.1}
Let $(\gamma _{n})_{n\ge 1}$ be the sequence of numbers defined by
\begin{equation}\label{defdellagnn}
\gamma_{n}=1+M\frac{\Gamma(n+\alpha)}{\Gamma(n)}.
\end{equation}
Write
\begin{equation}\label{defbetlagnn}
\beta_n=-\frac{\gamma_{n+1}}{\gamma_{n}},\quad n\ge 1.
\end{equation}
Then the polynomials defined by $q_0=1$ and
\begin{equation}\label{defqnlagnn}
q_{n}=L_{n}^{\alpha }+\beta _nL_{n-1}^{\alpha }, \quad n\ge 1,
\end{equation}
are orthogonal with respect to $\tilde \rho$ (\ref{darlag}).
\end{lemma}

\begin{proof}
In the notation of Lemma \ref{addel}, we have $\lambda=0$, $\nu =\mu _{\alpha -1}(x)dx$ and $\mu=\mu_{\alpha}(x)dx$.
Hence the formula $L_{n}^{\alpha+1}=\sum_{j=0}^nL_{n-j}^{\alpha}$ gives $\alpha_{n}=\Gamma(\alpha)$. Since we also have $p_n(0)=\binom{n+\alpha}{n}$,
an straightforward computation gives (\ref{defdellagnn}) and (\ref{defbetlagnn}) from (\ref{hk2}).
\end{proof}

The formula (\ref{defqnlagnn}) is the particular case $k=1$ of formula (3.16) in \cite{Plamen2}.

\bigskip

When $\alpha$ is a positive integer, the sequence $(\gamma_n)_n$ (\ref{defdellagnn}) is a polynomial
in $n$, and then Theorem \ref{th5.1lag} gives the  $(2\alpha +2)$-th order differential equation
that the Krall-Laguerre-Koornwinder polynomials (\ref{defqnlagnn}) satisfy.

\begin{theorem} \label{kkk} Let $(q_{n})_n$ be the sequence of orthogonal polynomials (\ref{defqnlagnn}) with respect to the measure $\tilde \rho$ (\ref{darlag}).
Assume that $\alpha$ is a positive integer.
Define the sequence of numbers $\lambda _n$ by
\begin{equation}\label{eiglagk}
\lambda _n=n+\frac{M}{\alpha +1}n(n+1)\cdots (n+\alpha ) ,
\end{equation}
and the polynomials $P_1$ and $P_2$, of degrees $\alpha+1$ and $\alpha$, respectively, by
\begin{align*}
P_1(x)&=-x+\frac{M}{\alpha +1}(-x)(-x+1)\cdots (-x+\alpha),\\
P_2(x)&=1+M(-x+1)(-x+2)\cdots (-x+\alpha).
\end{align*}
Consider the differential operator of order $2\alpha +2$ defined by
$$
D=P_1(D_{\alpha })+\frac{d}{dx} P_2(D_{\alpha }),
$$
where $D_{\alpha }$ is the second order differential operator for the Laguerre polynomials defined by (\ref{sodolag}). Then
$$
D(q_{n})=\lambda_nq_{n}.
$$
\end{theorem}

\begin{proof}

It is enough to check that $P_2(x)=P_1(x-1)-P_1(x)$ and then use Theorem \ref{th5.1lag}.

\end{proof}

Except for a sign, the eigenvalues $(\lambda_n)_n$ (\ref{eiglagk}) coincide with that given by J. and R. Koekoek in \cite{koekoe}. Hence,
we can conclude that the $(2\alpha+2)$-th order differential operator $-D$, where $D$ is displayed in the previous theorem, coincides
with the differential operator found by J. and R. Koekoek for the Krall-Laguerre-Koornwinder polynomials. This provides a nice expression for Koekoek
differential operator as a polynomial combination of the Laguerre second order differential operator.

\bigskip
When $\alpha$ is not equal to $k$, the polynomials $(q_n)_n$ (\ref{qnnolag}) in Theorem \ref{th5.1lag} seem not to be orthogonal
with respect to a measure. Anyway they enjoy certain orthogonality property. Indeed, when $P_2(1)\not =0$, since $P_2$ is a polynomial of degree $k$,
we can choose numbers $w_j$, $j=1,\cdots ,k$, such that
\begin{equation}\label{pqdlc2}
P_2(-x)=P_2(1)\left( 1+\sum_{j=1}^{k}w_j\binom{x+j}{j}\right).
\end{equation}
On the other hand, if $P_2(1) =0$, we can choose numbers $w_j$, $j=0,\cdots ,k$, such that
\begin{equation}\label{pqdlc3}
P_2(-x)= 1+\sum_{j=0}^{k}w_j\binom{x+j}{j}
\end{equation}
(in particular $w_0=-1$).

We then define the polynomial $Q$ by
\begin{equation}\label{pqdlc}
Q(x)=\begin{cases} \sum_{j=1}^k (\alpha-j)_jw_jx^{k-j}, &\mbox{if $P_2(1)\not =0$},\\
 \sum_{j=0}^k(\alpha-j)_jw_jx^{k-j}, &\mbox{if $P_2(1)=0$}.
 \end{cases}
\end{equation}
Notice that the polynomial $Q$ has degree at most $k-1$ for $P_2(1)\not =0$, but $k$ for $P_2(1)=0$.

\begin{lemma} \label{occ}
Let $P_2$ be an arbitrary polynomial of degree $k$ satisfying that $P_2(-n)\not =0$, $n\ge 0$.
For $\alpha \in \RR \setminus \{ k, k-1, \cdots \}$, consider the polynomials $(q_n)_n$ defined by (\ref{qnnolag}).
If we write
\begin{equation}\label{innlag}
\langle f,g\rangle=\int_0 ^\infty f(x)g(x)\mu _{\alpha -1}(x)dx+g(0)\int_0 ^\infty f(x)Q(x)\mu _{\alpha -k-1}(x)dx,
\end{equation}
where $\mu_\alpha$ is the orthogonalizing measure for the Laguerre polynomials $(L_{n}^{\alpha})_n$ and $Q$ is the polynomial defined by (\ref{pqdlc}),
then
\begin{enumerate}
\item $\langle q_n,q_j\rangle=0$, $j=0,\cdots ,n-1$,
\item $\langle q_n,q_n\rangle\not =0$, and
\item $\langle x^{k+1}q_n,q_j\rangle=0$, $n\ge k+2$ and $j=0,\cdots ,n-k-2$.
\end{enumerate}
\end{lemma}

\begin{proof}
It is enough to prove that $\langle q_n,x^j\rangle=0$, $j=0,\cdots ,n-1$,
$\langle q_n,x^n\rangle\not =0$, and $\langle x^{k+1}q_n,x^j\rangle=0$, $n\ge k+2$ and $j=0,\cdots ,n-k-2$.

Since $q_n$ is a linear combination of $L_{n}^{\alpha}$ and $L_{n-1}^{\alpha}$, we automatically have that
\begin{align*}
\langle q_n,x^j\rangle&=0,\quad j=1,\cdots ,n-1,\\
\langle x^{k+1}q_n,x^j\rangle&=0, \quad \mbox{$n\ge k+2$ and $j=0,\cdots ,n-k-2$},
\end{align*}
whatever the polynomial $Q$ and the numbers $\beta_n$, $n\ge 0$, are. In the same way, we have for $n\ge 1$
$$
\langle q_n,x^n\rangle =\int q_n(x)x^{n-1}\mu _{\alpha }(x)dx=(-1)^{n-1}(n-1)!\beta_n\Vert L_{n-1}^{\alpha}\Vert_2\not=0.
$$

For $j=0$, we proceed as follows. Using the expansion $L_{n}^{\alpha}=\sum_{j=0}^n\frac{(\alpha-\beta)_j}{j!}L_{n-j}^{\beta}$
(see \cite{EMOT}, p. 192, (39))  we get for $l\ge 1$
$$
\int L_{n}^{\alpha}\mu_{\alpha -l}(x)dx=\Gamma(\alpha-l+1)\binom{n+l-1}{l-1}.
$$
We first prove the Lemma when $P_2(1)\not =0$.
Using the definition of $Q$ (\ref{pqdlc}) and taking into account that $x^j\mu_\alpha=\mu_{\alpha+j}$, we get
$$
\langle q_n,1\rangle =\Gamma (\alpha)\left((1+\beta_n)
+\sum_{j=1}^{k}w_j\left[\binom{n+j}{j}+\beta_n\binom{n+j-1}{j}\right]\right).
$$

Equating $\langle q_n,1\rangle=0$, we get (using (\ref{pqdlc2}))
$$
\beta_n=-\frac{1+\sum_{j=1}^{k}w_j\binom{n+j}{j}}
{1+\sum_{j=1}^{k}w_j\binom{n+j-1}{j}}=-\frac{P_2(-n)}
{P_2(-(n-1))},
$$
which it is precisely the expression for $\beta_n$ in Theorem \ref{th5.1lag}.

Finally,
$$
\langle 1,1\rangle=\Gamma (\alpha)\left(1+\sum_{j=1}^{k}w_j\right)
=\Gamma (\alpha)P_2(0)/P_2(1)\not =0.
$$
When $P_2(1)=0$, the proof is similar.
\end{proof}

The orthogonality properties above for the polynomials $(q_n)_n$ imply that they have to satisfy
a $(2k+3)$-term recurrence relation.

\begin{corollary}
Assume that $\alpha \not =k, k-1, \cdots $, then the polynomials $(q_n)_n$ defined in Theorem \ref{th5.1lag} satisfy a $(2k+3)$-term
recurrence relation of the form
$$
x^{k+1}q_n(x)=\sum_{j=-k-1}^{k+1}c_{n,j}q_{n+j}(x).
$$
\end{corollary}

\begin{proof}
Write $x^{k+1}q_n=\sum_{j=-n}^{k+1}c_{n,j}q_{n+j}$. Then, part (3) of the previous Lemma gives
$$
0=\langle x^{k+1}q_n,q_0\rangle =c_{n,-n}\langle q_0,q_0\rangle.
$$
Using part (2) of the previous Lemma we get $c_{n,-n}=0$. We now proceed by induction on $m$. Assume $c_{n,-n+j}=0$, $j=0,\cdots, m$, $m<n-k-2$.
Proceeding as before, we find
$$
0=\langle x^{k+1}q_n,q_{-n+m+1}\rangle =c_{n,-n+m+1}\langle q_{-n+m+1},q_{-n+m+1}\rangle,
$$
from where we deduce that also $c_{n,-n+m+1}=0$.

\end{proof}

\subsection{Jacobi polynomials}

For $\alpha,\beta \in \RR , \alpha+\beta \not=-1,-2,\cdots$, we use the standard definition of the Jacobi polynomials $(P_{n}^{\alpha,\beta})_n$
\begin{equation}\label{defjac}
P_{n}^{\alpha,\beta}(x)=2^{-n}\sum _{j=0}^n \binom{n+\alpha}{j}\binom{n+\beta}{n-j}(x-1)^{n-j}(x+1)^{j}
\end{equation}
(that and the next formulas can be found in \cite{EMOT}, pp. 169-173).

For $\alpha,\beta, \alpha+\beta \not =-1,-2,\cdots$,
they are orthogonal with respect to a measure $\mu _{\alpha,\beta}=\mu _{\alpha,\beta}(x)dx$, which it is positive only when
$\alpha ,\beta >-1$, and then
\begin{equation}\label{jacw}
\mu_{\alpha,\beta}(x) =(1-x)^\alpha (1+x)^{\beta}, \quad -1<x<1.
\end{equation}
They are eigenfunctions of the following second order differential operator
\begin{align}\label{sodojac}
D_{\alpha,\beta} &=(1-x^2)\left(\frac{d}{dx}\right) ^2+(\beta-\alpha-(\alpha+\beta+2)x)\frac{d}{dx},\\ \nonumber
D_{\alpha,\beta} &(P_{n}^{\alpha,\beta})=-n(n+\alpha+\beta+1)P_{n}^{\alpha,\beta},\quad n\ge 0.
\end{align}

We omit the proofs in this Section, since they are essentially the same as in the previous Sections.

We first construct two $\D$-operators for the Jacobi polynomials.

\begin{lemma}\label{lTjac}
For $\alpha, \beta \in \RR$ satisfying $\alpha+\beta \not =-1, -2, \cdots $, let $p_{n}=P_{n}^{\alpha ,\beta}$, $n\ge 0$, the Jacobi polynomials given by (\ref{defjac}). Then the operators $\D_1$ and $\D _2$ defined by (\ref{defTo2}) from the sequences $(n\ge 0)$
\begin{align*}
\varepsilon _{n,1}&=\frac{n+\alpha}{n+\alpha+\beta},\quad &\sigma_{n,1}&=2n+\alpha +\beta-1 ,\\
\varepsilon _{n,2}&=-\frac{n+\beta}{n+\alpha+\beta},\quad &\sigma_{n,2}&=-(2n+\alpha +\beta-1) ,
\end{align*}
are $\D$-operator for the Jacobi polynomials and the algebra $\A$  (\ref{algdiff1}).
More precisely
\begin{align*}
\D_1&=\displaystyle -\frac{\alpha +\beta+1}{2}I+(1-x)\frac{d}{dx},\\
\D_2&=\displaystyle \frac{\alpha +\beta+1}{2}I+(1+x)\frac{d}{dx}.
\end{align*}
\end{lemma}

To establish the main results in this Section, we need some notation.

We write $(\sigma_n)_n$ and $(\theta_n)_n$ for the sequences
\begin{align*}\nonumber
\theta_n&=-n(n+\alpha +\beta+1), \\
\sigma_n&=2n+\alpha +\beta-1.
\end{align*}
We recall that $(\theta_n)_n$ is just the  eigenvalue sequence for the Jacobi polynomials $P_{n}^{\alpha,\beta}$ (\ref{defjac}).
For $j\ge 0$, we also consider the sequence
$$
u_j(n)=(n+\alpha )_j(n+\beta -j)_j,
$$
and the polynomials $r_0=1$ and for $j\ge 1$
$$
r_j(x)=\prod _{i=0}^{j-1}[-x+(\alpha+i+1)(\beta -i)].
$$

Since the polynomial $r_j$ has degree just $j$, all polynomial
$P\in  \PP$ can be written as $P(x)=\sum_{j=0}^{\deg(P)}w_jr_j(x)$ for certain numbers $w_j$, $j=0,\cdots , \deg(P)$.

Lemma \ref{fl2v} and the two $\D$-operators for the Jacobi polynomials
found in Lemma \ref{lTjac} automatically produce  a large class of
polynomials satisfying higher order differential equations.
We only proceed with the first $\D$-operator because due to
the symmetry of the Jacobi polynomials with respect to the parameters $\alpha$ and $\beta$ and the endpoints of its support $x=1$ and $x=-1$
the examples generated by both $\D$-operators are essentially the same.

\begin{theorem}\label{th5.1jac}
Let $P_2$ be an arbitrary polynomial of degree
 $k\ge 1$ which we write in the form
$$
P_2(x)=\sum_{j=0}^{k}w_jr_j(x),
$$
for certain numbers $w_j$, $j=0,\cdots , k$. Consider also the polynomial $P_1$ of degree $k+1$ defined by
$$
P_1(x)=\sum _{j=0}^k w_j\left( \frac{2}{j+1}r_{j+1}(x)+(\alpha-\beta+2j+1)r_j(x)\right).
$$
Consider next the sequences of numbers $(\gamma_n)_n$, $(\beta_n)_n$ and $(\lambda_n)_n$ defined by
\begin{align*}\nonumber
\gamma_n&=\sum _{j=0}^k w_ju_j(n),\quad n\ge 0,\\
\beta_n&=\frac{n+\alpha}{n+\alpha+\beta}\frac{\gamma_{n+1}}{\gamma_n},\quad n\ge 1,\\
 \lambda_n&=\frac{\sigma _n\gamma_n+P_1(\theta_{n-1})}{2}, \quad n\ge 1,\quad \lambda_0=\frac{P_1(0)-(\alpha+\beta+1)P_2(0)}{2}.
\end{align*}
where we implicitly assume that $\gamma_n\not =0$, $n\ge 0$. Define finally the sequence of polynomials $q_0=1$, and for $n\ge 1$
\begin{equation}\label{qnnojac}
q_n=P_{n}^{\alpha,\beta}+\beta_nP_{n-1}^{\alpha,\beta},
\end{equation}
where $P_{n}^{\alpha,\beta}$ is the Jacobi polynomial (\ref{defjac}) ($\alpha+\beta \not =-1, -2,\cdots$).
Write finally $D$ for the differential operator of order $2k+2$
$$
D=\frac{1}{2}P_1(D_{\alpha,\beta})+\left(-\frac{\alpha+\beta+1}{2}+(1-x)\frac{d}{dx}\right) P_2(D_{\alpha,\beta}),
$$
where $D_{\alpha,\beta}$  is the second order differential operator  with respect to which
the polynomials $(P_{n}^{\alpha,\beta})_n$ are eigenfunctions (see (\ref{sodojac})).

Then $D(q_n)=\lambda_nq_n$, $n\ge 0$.
\end{theorem}

The case $k=1$ in the previous Theorem is included in \cite{GrH3}, and for general $k$ is implicit in \cite{Plamen1}.

We next consider some important particular cases of the previous Theorem.

For $\alpha,\beta,\alpha+\beta \not =-1,-2, \cdots $ and $M\not =-\Gamma(n)\Gamma(n+\alpha)/(\Gamma(n+\alpha+\beta)\Gamma(n+\beta))$, consider
the measure $\tilde \rho_{\alpha,\beta,M}$ defined by
\begin{equation}\label{darjac}
\tilde \rho_{\alpha,\beta,M}=2^{\alpha+\beta}\beta \Gamma^2(\beta )M\delta _{-1}+\mu_{\alpha,\beta-1}(x)dx,
\end{equation}
where $\mu_{\alpha,\beta}(x)$ is the orthogonalizing function for the Jacobi polynomials $(P_{n}^{\alpha,\beta})_n$.
When $\alpha>-1, \beta>0$ and $M>0$, the measure $\tilde \rho _{\alpha,\beta, M}$ is the
Krall-Jacobi-Koornwinder measure
$$
2^{\alpha+\beta}\beta \Gamma^2(\beta )M\delta _{-1}+(1-x)^\alpha (1+x)^{\beta-1}, \quad -1<x<1.
$$
To simplify the notation we remove the dependence of $\alpha$, $\beta$ and $M$ and write $\tilde \rho =\tilde \rho_{\alpha,\beta, M}$.

Using Lemma \ref{addel} we find explicitly a sequence of orthogonal polynomials with respect to $\tilde  \rho$ in terms of $p_n=P_{n}^{\alpha,\beta}$, $n\ge 0$.

\begin{lemma}
Let $(\gamma _{n})_{n\ge 1}$ be the sequence of numbers defined by
\begin{equation}\label{defdeljac}
\gamma_{n}=1+M\frac{\Gamma(n+\alpha+\beta)}{\Gamma(n+\alpha)}\frac{\Gamma(n+\beta)}{\Gamma(n)}.
\end{equation}
Write
\begin{equation*}
\beta_n=\frac{n+\alpha}{n+\alpha+\beta}\frac{\gamma_{n+1}}{\gamma_{n}},\quad n\ge 1.
\end{equation*}
Then the polynomials $q_0=1$ and
\begin{equation}\label{defqnjac}
q_{n,}=P_{n}^{\alpha,\beta }+\beta _nP_{n-1}^{\alpha,\beta }, \quad n\ge 1,
\end{equation}
are orthogonal with respect to $\tilde \rho$ (\ref{darjac}).
\end{lemma}

The formula (\ref{defqnlagnn}) is the particular case $k=1$ of formula (3.10) in \cite{Plamen1}.

\bigskip

When $\beta$ is a positive integer, the sequence $(\gamma_n)_n$ (\ref{defdeljac}) is a polynomial
in $n$, and then Theorem \ref{th5.1lag} gives the  $(2\alpha +2)$-th order differential equation
that the Krall-Jacobi-Koornwinder polynomials (\ref{defqnjac}) satisfy.

\begin{theorem} Let $(q_{n})_n$ be the sequence of polynomials (\ref{defqnjac}) orthogonal with respect to the measure $\tilde \rho$ (\ref{darjac}).
Assume that $\beta$ is a positive integer.
Define the sequence of numbers $\lambda _n$ by
\begin{align} \label{eigjack}
\lambda _n&=(n+\alpha+\beta)\left(n+M(n+\alpha)_\beta(n)_\beta\left(1+\frac{n-1}{\beta +1}\right)\right)+\alpha \beta,
\end{align}
and the polynomials $P_1$ and $P_2$, of degrees $\beta+1$ and $\beta$, respectively, by
\begin{align*}
P_1(x)&=\alpha \beta+(\alpha +1)(\beta+1)-2x+M\left(-\frac{2x}{\beta +1}+\alpha+\beta +1\right) \\ \nonumber
&\qquad \qquad \quad \times \prod_{i=0}^{\beta-1}[-x+(\alpha+i+1)(\beta-i)],\\
P_2(x)&=1+M\prod_{i=0}^{\beta-1}[-x+(\alpha+i+1)(\beta-i)].
\end{align*}
Consider the differential operator of order $2\beta +2$ defined by
$$
D=\frac{1}{2}P_1(D_{\alpha,\beta })+\left(-\frac{\alpha+\beta+1}{2}+(1-x)\frac{d}{dx} \right)P_2(D_{\alpha,\beta}),
$$
where $D_{\alpha,\beta }$ is the second order differential operator defined by (\ref{sodojac}). Then
$$
D(q_{n})=\lambda_nq_{n}.
$$
\end{theorem}

Comparing the eigenvalues  $(\lambda_n)_n$ (\ref{eigjack}) with that given by Zhedanov in \cite{Zh}, (8.14),
we can conclude that the differential operator found by Zhedanov for the Krall-Jacobi-Koornwinder polynomials orthogonal with respect to (\ref{jacobid}), $N=0$,
coincides (up to normalization) with the $(2\beta+2)$-th order differential operator displayed in the previous theorem.

\bigskip
When $\beta$ is not equal to $k$, the polynomials $(q_n)_n$ (\ref{qnnojac}) in Theorem \ref{th5.1jac} seem not to be orthogonal with respect to a measure. Proceeding
in the same way as in the previous Section (see Lemma \ref{occ}) one can see that they enjoy similar orthogonality properties
with respect to the nonsymmetric bilinear form
$$
\langle f,g\rangle=\int_{-1} ^1 f(x)g(x)\mu _{\alpha,\beta -1}(x)dx+g(-1)\int_{-1} ^1 f(x)Q(x)\mu _{\alpha,\beta -k-1}(x)dx,
$$
where $\mu_{\alpha,\beta}$ is the orthogonalizing measure for the Jacobi polynomials $(P_{n}^{\alpha,\beta })_n$,
and $Q$ is certain polynomial of degree at most $k-1$ or $k$ depending on $P_2$.
These orthogonality properties for the polynomials $(q_n)_n$ imply that they satisfy a $(2k+3)$-term recurrence relation of the form
$$
(1+x)^{k+1}q_n(x)=\sum_{j=-k-1}^{k+1}c_{n,j}q_{n+j}(x).
$$

     \end{document}